\documentclass[a4paper,oneside,11pt]{article}

\usepackage[a4paper]{geometry}
\usepackage{aeguill}                 
\usepackage{graphicx}
\usepackage{amsmath,amsfonts,amssymb,amsthm}
\usepackage{pstricks} 
\usepackage{epstopdf}
\usepackage{srcltx}
\usepackage[utf8]{inputenc} 
\usepackage[T1]{fontenc} 
\usepackage{hyperref} 
\usepackage[english]{babel} 
\usepackage{comment}
\usepackage{enumerate}

\title{Measure equivalence rigidity of the handlebody groups}
\author{Sebastian Hensel and Camille Horbez}

\usepackage[all]{xy}
\usepackage{bbm}
\begin{document}
\maketitle
\newtheorem{de}{Definition} [section]
\newtheorem{theo}[de]{Theorem} 
\newtheorem{prop}[de]{Proposition}
\newtheorem{lemma}[de]{Lemma}
\newtheorem{lem}[de]{Lemma}
\newtheorem{cor}[de]{Corollary}
\newtheorem{propd}[de]{Proposition-Definition}
\newtheorem{conj}[de]{Conjecture}
\newtheorem{claim}{Claim}
\newtheorem*{claim2}{Claim}

\newtheorem{theointro}{Theorem}
\newtheorem*{defintro}{Definition}
\newtheorem{corintro}[theointro]{Corollary}
\newtheorem{propintro}[theointro]{Proposition}

\theoremstyle{remark}
\newtheorem{rk}[de]{Remark}
\newtheorem{ex}[de]{Example}
\newtheorem{question}[de]{Question}

\normalsize

\newcommand{\cala}{\mathcal{A}}
\newcommand{\calc}{\mathcal{C}}
\newcommand{\calf}{\mathcal{F}}
\newcommand{\calg}{\mathcal{G}}
\newcommand{\calh}{\mathcal{H}}
\newcommand{\calp}{\mathcal{P}}
\newcommand{\calk}{\mathcal{K}}
\newcommand{\calb}{\mathcal{B}}
 
\newcommand{\pstab}{\mathrm{PStab}} 
\newcommand{\Mod}{\mathrm{Mod}}
\newcommand{\DA}{\mathbb{DA}}
\newcommand{\crs}{\mathrm{CRS}}
\newcommand{\Prob}{\mathrm{Prob}}
\newcommand{\PMF}{\mathrm{PMF}}
\newcommand{\Aut}{\mathrm{Aut}}
\newcommand{\Stab}{\mathrm{Stab}}
\newcommand{\Cay}{\mathrm{Cay}}
\newcommand{\nsep}{\mathrm{nsep}}
\newcommand{\sep}{\mathrm{sep}}
\newcommand{\spec}{\mathrm{spec}}
\newcommand{\actson}{\curvearrowright}
\newcommand{\uni}{\mathrm{uni}}
\newcommand{\dunion}{\sqcup}
\newcommand{\psep}{(\mathrm{P}_{\mathrm{sep}})}
\newcommand{\pnsep}{(\mathrm{P}_{\nsep})}
\newcommand{\qnsep}{(\mathrm{Q}_{\nsep})}

\newcommand{\ND}{\mathbb{D}^{\nsep}(V)}
\newcommand{\NDA}{\mathbb{DA}^{\uni}}

\newcommand{\PML}{\mathrm{PML}}
\newcommand{\NAL}{\mathrm{NAL}}
\newcommand{\AL}{\mathrm{AL}}
 
\newcommand{\Ccom}[1]{\Cmod\marginpar{\color{red}\tiny #1 --ch}} 
\newcommand{\Scom}[1]{\Smod\marginpar{\color{blue}\tiny #1 --sh}}
\newcommand{\Cmod}{$\textcolor{red}{\clubsuit}$} 
\newcommand{\Smod}{$\textcolor{blue}{\spadesuit}$}

\begin{abstract}
Let $V$ be a connected $3$-dimensional handlebody of finite genus at least $3$. We prove that the handlebody group $\Mod(V)$ is superrigid for measure equivalence, i.e.\ every countable group which is measure equivalent to $\Mod(V)$ is in fact virtually isomorphic to $\Mod(V)$. Applications include a rigidity theorem for lattice embeddings of $\Mod(V)$, an orbit equivalence rigidity theorem for free ergodic measure-preserving actions of $\Mod(V)$ on standard probability spaces, and a $W^*$-rigidity theorem among weakly compact group actions.
\end{abstract}

\section*{Introduction}
 
A central quest in measured group theory is to classify countable groups up to \emph{measure equivalence}, a notion coined by Gromov in \cite{Grom} as a measurable analogue to the geometric notion of quasi-isometry between finitely generated groups. 

The definition is as follows: two countable groups $\Gamma_1$ and $\Gamma_2$ are \emph{measure equivalent} if there exists a standard measure space $\Omega$ (of positive measure) equipped with an action of $\Gamma_1\times\Gamma_2$ by measure-preserving Borel automorphisms, such that for every $i\in\{1,2\}$, the action of $\Gamma_i$ on $\Omega$ is free and has a fundamental domain of finite measure. The typical example is that any two (possibly non-uniform) lattices in the same locally compact second countable group $G$ are always measure equivalent, by considering the left and right multiplications on $G$ equipped with its Haar measure.

Dye proved in \cite{Dye1,Dye2} that all countably infinite abelian groups are measure equivalent. This was famously generalized by Ornstein and Weiss to all countably infinite amenable groups \cite{OW}, and in fact these form a class of the measure equivalence relation on the set of all countably infinite groups, see \cite[Corollary~1.3]{Fur2}. At the other extreme of the picture, some groups satisfy very strong rigidity properties. A first striking example is the following: building on earlier work of Zimmer \cite{Zim1,Zim2}, Furman proved that every countable group which is measure equivalent to a lattice in a center-free higher rank simple Lie group, is commensurable to a lattice in the same Lie group up to a finite kernel \cite{Fur2}. In \cite{MS}, Monod and Shalom proved superrigidity type results for direct products of groups that satisfy an analytic form of negative curvature, phrased in terms of a bounded cohomology criterion. Later, Kida proved that, with the exception of some low-complexity cases, mapping class groups $\Mod(\Sigma)$ of orientable finite-type surfaces are \emph{ME-superrigid}, i.e.\ every countable group that is measure equivalent to $\Mod(\Sigma)$, is in fact commensurable to $\Mod(\Sigma)$ up to a finite kernel \cite{Kid}. This led to further strong rigidity results, for certain amalgamated free products \cite{Kid2}, certain subgroups of $\Mod(\Sigma)$ such as the Torelli group \cite{CK}, some infinite classes of Artin groups of hyperbolic type \cite{HoHu}. Very recently, Guirardel and the second-named author established that $\mathrm{Out}(F_N)$, the outer automorphism group of a finitely generated free group of rank $N\ge 3$, is also ME-superrigid \cite{GH}. 

In the present paper, we establish a superrigidity theorem for
\emph{handlebody groups}, defined as mapping class groups $\Mod(V)$ of
connected $3$-dimensional handlebodies $V$, i.e.\ $V$ is a disk-sum of
finitely many copies of $D^2\times S^1$. These groups are of
particular importance in $3$-dimensional topology, and most notably in
the theory of Heegaard splittings, see e.g.\ the discussion in
\cite[Section~4]{Hen-primer}. They are also important in geometric
group theory due to their direct connections to both mapping
class groups of surfaces and outer automorphism groups of free
groups. Notice indeed that $\partial V$ is a closed orientable surface
of finite genus $g\ge 0$, and $\Mod(V)$ embeds as a (highly distorted
\cite{HH1}) subgroup of $\Mod(\partial V)$; it also surjects onto
$\mathrm{Out}(F_g)$ via the action at the level of the fundamental
group (with non-finitely generated kernel \cite{McCul}). Recently, 
the geometry of handlebody groups has been shown to share many features with
outer automorphism groups of free groups rather than surface mapping
class groups (e.g. concerning the growth of isoperimetric functions \cite{HH}
or the subgroup geometry of stabilisers \cite{HenStab}).

Handlebody groups are known to satisfy some algebraic rigidity
properties. Let $\Mod^\pm(V)$ be the \emph{extended handlebody group}, where one allows orientation-reversing homeomorphisms. 
Korkmaz and Schleimer proved in \cite{KS} that the outer
automorphism group of $\Mod^\pm(V)$ is trivial, and the first-named author further
proved in \cite{Hen} that the natural map from $\Mod^\pm(V)$ to
its abstract commensurator is an isomorphism. To our knowledge, the
question of the quasi-isometric rigidity of handlebody groups (which
are finitely generated by work of Suzuki \cite{Suz}, in fact finitely
presented by work of Wajnryb \cite{Waj}) is still widely open. Our
main theorem establishes their superrigidity from the viewpoint of
measured group theory.

\begin{theointro}\label{theointro:main}
Let $V$ be a connected $3$-dimensional handlebody of finite genus at least $3$. Then $\Mod(V)$ is ME-superrigid. 
\end{theointro}

\paragraph*{Consequences.} The techniques used in the proof of Theorem~\ref{theointro:main} have several other consequences. First, we recover (with a different argument) the commensurator rigidity statement established by the first named author in \cite{Hen}, see Remark~\ref{rk:commensurator}.

Second, using ideas of Furman \cite{Fur} and Kida \cite{Kid}, we derive that handlebody groups cannot embed as lattices in second countable locally compact groups in any interesting way. 

\begin{corintro}\label{corintro:lattice}
Let $V$ be a connected $3$-dimensional handlebody of finite genus at least $3$. Let $G$ be a locally compact second countable group, equipped with its Haar measure. Let $\Gamma$ be a finite index subgroup of $\Mod^\pm(V)$, and let $\sigma:\Gamma\to G$ be an injective homomorphism whose image is a lattice. 

Then there exists a homomorphism $\theta:G\to\Mod^\pm(V)$ with compact kernel such that for every $f\in\Gamma$, one has $\theta\circ\sigma(f)=f$.
\end{corintro}
If $S$ is a finite generating set of $\Mod^\pm(V)$, then $\Mod^\pm(V)$ naturally embeds as a lattice in the automorphism group of the Cayley graph $\Cay(\Mod^\pm(V),S)$, defined as the simplicial graph whose vertices are the elements of $\Mod^\pm(V)$, with an edge between two distinct vertices $g,h$ whenever $gh^{-1}\in S\cup S^{-1}$ (this convention excludes for instance loop-edges when $S$ contains the identity of $\Mod^\pm(V)$, or multiple edges if $S$ contains an element and its inverse). The above rigidity statement about lattice embeddings has the following consequence (which can also be viewed as a very weak form of the conjectural quasi-isometry rigidity statement).

\begin{corintro}
Let $V$ be a connected $3$-dimensional handlebody of finite genus at least $3$, and let $S$ be a finite generating set of $\Mod^\pm(V)$. Then every graph automorphism of $\Cay(\Mod^\pm(V),S)$ is at bounded distance from the left multiplication by an element of $\Mod^\pm(V)$. 

If $\Gamma$ is a torsion-free finite-index subgroup of $\Mod^\pm(V)$, and if $S'$ is a finite generating set of $\Gamma$, then the automorphism group of $\Cay(\Gamma,S')$ is countable, and in fact embeds as a subgroup of $\Mod^\pm(V)$ containing $\Gamma$.
\end{corintro}

The torsion-freeness assumption is crucial in the second part of the statement: for every finitely generated group $G$ containing a nontrivial torsion element, there exists a finite generating set $S$ of $G$ such that the automorphism group of $\Cay(G,S)$ is uncountable, as was observed by de la Salle and Tessera in \cite[Lemma~6.1]{dlST}.  

Thanks to work of Furman \cite{Fur3}, the measure equivalence rigidity statement given in Theorem~\ref{theointro:main} can also be recast in the language of orbit equivalence rigidity of probability measure-preserving ergodic group actions. We reach the following corollary, analogous to a theorem of Kida \cite{Kid2} for mapping class groups -- see Section~\ref{sec:oe} for all definitions.

\begin{corintro}\label{corintro:oe}
Let $V$ be a connected $3$-dimensional handlebody of finite genus at least $3$. Let $\Gamma$ be a countable group. Let $\Mod^\pm(V)\actson X$ and $\Gamma\actson Y$ be two free ergodic measure-preserving group actions by Borel automorphisms on standard probability spaces. 

If the actions $\Mod^\pm(V)\actson X$ and $\Gamma\actson Y$ are stably orbit equivalent, then they are virtually conjugate.
\end{corintro}

Finally, our work also yields strong rigidity statements for von Neumann algebras associated (via a celebrated construction of Murray and von Neumann \cite{MvN}) to probability measure-preserving ergodic group actions of handlebody groups. By combining Corollary~\ref{corintro:oe} with the \emph{proper proximality} of handlebody groups in the sense of Boutonnet, Ioana and Peterson \cite{BIP} (established in \cite{HHL}), we reach the following corollary -- see Section~\ref{sec:oe} for definitions, and work of Ozawa and Popa \cite[Definition~3.1]{OP} for the notion of a weakly compact group action (as an important example, the action of a residually finite group on its profinite completion is weakly compact). 

\begin{corintro}\label{theointro:von-neumann}
Let $V$ be a connected $3$-dimensional handlebody of finite genus at least $3$. Let $\Gamma$ be a countable group. Let $\Mod^\pm(V)\actson X$ and $\Gamma\actson Y$ be two free ergodic measure-preserving group actions by Borel automorphisms on standard probability spaces, and assume that $\Gamma\actson Y$ is weakly compact. 

If the von Neumann algebras $L^\infty(X)\rtimes\Mod^\pm(V)$ and $L^\infty(Y)\rtimes\Gamma$ are isomorphic, then the actions $\Mod^\pm(V)\actson X$ and $\Gamma\actson Y$ are virtually conjugate.
\end{corintro}

\paragraph*{Proof strategy.} The general strategy of our proof of Theorem~\ref{theointro:main} follows Kida's approach for mapping class groups \cite{Kid}. General techniques from measured group theory, originating in the work of Furman \cite{Fur2}, reduce the proof of Theorem~\ref{theointro:main} to a cocycle rigidity theorem (Theorem~\ref{theo:main-2}) for actions of $\Mod(V)$ on standard probability spaces. In order to avoid some finite-order phenomena, it is in fact useful for us to work in a finite-index rotationless subgroup $\Mod^1(V)$ (see Section~\ref{sec:rotationless} for its precise definition). More precisely, we are given a measured groupoid $\mathcal{G}$, which comes from restricting two actions of $\Mod^1(V)$ on standard finite measure spaces to a positive measure Borel subset $Y$ on which their orbits coincide. The groupoid $\calg$ is thus equipped with two cocycles $\rho_1,\rho_2:\calg\to\Mod^1(V)$, given by the two actions: whenever two points $x,y\in Y$ are joined by an arrow $g\in\calg$, there is an element $\rho_1(g)$ sending $x$ to $y$ for the first action, and an element $\rho_2(g)$ sending $x$ to $y$ for the second action. Our goal is to build a canonical map $\varphi:Y\to\Mod^\pm(V)$ such that $\rho_1$ and $\rho_2$ are cohomologous through $\varphi$: this means that whenever $x,y\in Y$ are joined by an arrow $g\in\calg$, then $\rho_2(g)=\varphi(y)\rho_1(g)\varphi(x)^{-1}$. In fact, using a theorem of Korkmaz and Schleimer which identifies $\Mod^\pm(V)$ to the automorphism group of the disk graph $\mathbb{D}$ of $V$, our goal is to build a (canonical) map $Y\to\Aut(\mathbb{D})$. Recall that the \emph{disk graph} is the graph whose vertices are the isotopy classes of \emph{meridians} in $\partial V$ (i.e.\ essential simple closed curves that bound a properly embedded disk in $V$), and two vertices are joined by an edge if the corresponding meridians have disjoint representatives in their respective isotopy classes.

In order to build the desired map $Y\to\Aut(\mathbb{D})$, the main step is to characterize subgroupoids of $\calg$ that arise as stabilizers of Borel maps $Y\to\mathbb{D}$ in a purely groupoid-theoretic way, i.e.\ with no reference to the cocycles (so that a vertex stabilizer for $\rho_1$ is also a vertex stabilizer for $\rho_2$). 

In the surface mapping class group setting (where the disk graph is replaced by the curve graph of the surface $\Sigma$), the important observation made by Kida is the following: curve stabilizers inside (a suitable finite index subgroup of) $\Mod(\Sigma)$ are characterized as maximal nonamenable subgroups which contain an infinite amenable normal subgroup (namely, the cyclic subgroup generated by the twist about the curve). This has a groupoid-theoretic analogue, through notions of amenable and normal subgroupoids.

The situation is more complicated for handlebodies, and the above algebraic statement does not give a characterization of meridian stabilizers any longer, for several reasons that we will now explain; for simplicity we will sketch the group-theoretic version of our arguments, but in reality everything has to be phrased in the language of measured groupoids. Our most challenging task, which occupies a large part of Section~\ref{sec:me}, is in fact to characterize stabilizers of nonseparating meridians. Inspired by the surface setting, we want to start with a maximal nonamenable subgroup $H$ of $\Mod^1(V)$ which contains an infinite amenable normal subgroup $A$. A first bad situation we encounter is the following: $A$ could be generated by a partial pseudo-Anosov, supported on a subsurface $S\subseteq\partial V$, and $H$ be its normalizer. In the surface setting considered by Kida \cite{Kid}, such an $H$ is not maximal, as it is contained in the stabilizer $H'$ of the boundary multicurve $\gamma$ of $S$. But for us, the group of multitwists about $\gamma$ could intersect $\Mod^1(V)$ trivially; in this case $H'$ may not contain any infinite normal amenable subgroup, so $H'$ may not violate the maximality of $H$. We resolve this first difficulty by further imposing that $H$ should not be contained in a subgroup containing two normal nonamenable subgroups that centralize each other (typically, the stabilizers of a subsurface and its complement); this is why we need to exclude separating meridians from our analysis at first. With a bit more work, we manage to reduce to the case where the pair $(H,A)$ is given by the following situation: there is a multicurve $X$, together with a (possibly empty) collection $\mathfrak{A}$ of complementary components of $X$ labeled active, $H$ is the stabilizer of $X$, and $A$ is exactly the active subgroup of $(X,\mathfrak{A})$, i.e.\ the subgroup of the stabilizer of $X$ acting trivially on all inactive subsurfaces, and it is amenable. This still includes several possibilities: $X$ could be a nonseparating meridian and $\mathfrak{A}=\emptyset$ (in which case $A$ is the twist subgroup). But (still with $\mathfrak{A}=\emptyset$), the multicurve $X$ could also be of the form $\alpha_1\cup\alpha_2$, where $\alpha_1$ and $\alpha_2$ together bound an annulus in $V$ (see Figure~\ref{fig:meridian}): the cyclic subgroup generated by the product of twists $T_{\alpha_1}T_{\alpha_2}^{-1}$ is then normal in the handlebody group stabilizer of the annulus. To exclude annuli (and in fact only retain nonseparating meridians), we use a combinatorial argument: roughly, we can always complete a nonseparating meridian to a collection of $3g-3$ such, while doing this with annulus pairs will introduce redundancy, as the same curves will be used more than once. Combinatorially, in a collection of $3g-3$ annuli, it is always possible to remove one without changing the link of the collection in an appropriate graph of disks and annuli. 

Once we have characterized nonseparating meridians, we actually have enough information to also recover the separating ones, exploiting that these can be completed to a pair of pants decomposition by adding $3g-4$ nonseparating meridians. Finally, a characterization of adjacency in the disk graph comes from observing that two meridians are disjoint up to isotopy if the corresponding twists commute, or in other words if these twists together generate an amenable subgroup of $\Mod(V)$.

\paragraph*{Acknowledgments.} We would like to thank an anonymous referee for a very careful reading and numerous comments which improved the paper. The first-named author is partially supported by the DFG as part of the SPP 2026 ``Geometry at Infinity''. The second-named author acknowledges support from the Agence Nationale de la Recherche under Grant ANR-16-CE40-0006 DAGGER.

\section{Handlebody and mapping class group facts}

In this section, we collect a few facts about handlebody groups that
will be useful in the paper. The reader is refered to \cite{FM} for an
introduction to mapping class groups and \cite{Johannson, Hen-primer}
for general information about handlebody groups.

\subsection{Mapping Class Group background}\label{sec:bg}
Let $\Sigma$ be a surface obtained from an oriented compact surface by
removing a finite number of points and open disks (a \emph{finite type
  surface}).  A \emph{essential simple closed curve (or simply curve)}
is an embedded copy of $S^1$ in $\Sigma$ which is not homotopic to a
point, a puncture, or a boundary component. Usually we do not
distinguish between isotopic curves. 

The \emph{extended mapping class group} $\Mod^\pm(\Sigma)$ is the
group of isotopy classes of homeomorphisms of $\Sigma$, and the \emph{mapping class group} $\Mod(\Sigma)$ is the subgroup formed by the orientation preserving mapping classes. We refer the
reader to e.g. \cite{FM} for basic facts on curves, their minimal
position, subsurfaces and basic mapping class facts. In this section
we only recall a few results which are particularly pertinent for our
purposes.

\paragraph*{Rotationless Mapping Classes.}
In order to avoid finite-order phenomena, it will be
useful to work in certain finite index subgroups of the mapping class
group. We say that a mapping class $f$ is \emph{rotationless} (or
\emph{pure}) if the following holds: if a power of $f$ fixes the
isotopy class of a simple closed curve $c$, then $f$ actually fixes
the oriented isotopy class of $c$ (in particular $f$ does not swap the two sides of $c$, e.g.\ if $c$ is separating, then $f$ preserves both complementary components).

We denote by $\Mod^0(\Sigma)$ the (finite index) subgroup of
$\Mod(\Sigma)$ consisting of mapping classes which preserve each
complementary component, and act trivially on homology mod $3$ of the
surface. We have
\begin{lemma}\label{lem:ivanov}
  \begin{itemize}
  \item Every $f\in\Mod^0(\Sigma)$ is rotationless.
  \item Every subgroup $G < \Mod^0(\Sigma)$ either contains a free
    group on two generators, or is free abelian.
  \end{itemize}
\end{lemma}
\begin{proof}
  The first claim is \cite[Theorem~1.2]{Iva-book} (noting that
  $H_1(S)$ in that source denotes homology with mod-$k$ coefficients
  with $k\geq 3$). The second claim is \cite[Theorem~8.9]{Iva-book}
  (noting that $\Gamma_S(m_0)$ is the group acting trivially on
  homology with mod-$m_0$ coefficients).
\end{proof}

\paragraph*{Subsurfaces.}

We emphasise that a \emph{subsurface $X \subset \Sigma$} is an
embedded copy of a finite type surface in $\Sigma$, so that every
boundary curve of $X$ is essential. Furthermore, we require that if
$A \subset X$ is an annulus, then it may not be homotopic into another
component of $X$.

As with curves, we usually do not distinguish between isotopic
subsurfaces. We note that it nevertheless makes sense to
\emph{intersect subsurfaces $X$ and $Y$}: the intersection is the (up
to isotopy unique) subsurface $X \cap Y$ with the property that a
simple closed curve is homotopic into $X \cap Y$ exactly if it is
homotopic into both $X$ and $Y$. To see that this is well-defined,
assume that $X$ and $Y$ are connected, and that $\Sigma$ carries a
hyperbolic metric (the remaining cases are a straightforward extension).
The existence and uniqueness is the clear
if $X$ or $Y$ is an annulus (in which case the intersection is empty,
or that annulus). In the non-annular case, take the unique representatives of
$X, Y$ with geodesic boundary components. Then the intersection of
those representatives has the desired property. 

\begin{lemma}\label{lem:subsurface-chains}
  Suppose that $\Sigma$ is a finite type surface. Then there is a
  number $N=N(\Sigma)$ so that if
  \[ S_1 \subset S_2 \subset \cdots \subset S_n = S \] is a chain of
  subsurfaces, so that $S_i$ is not isotopic to $S_{i+1}$ for any
  $i$, then $n\leq N$.
\end{lemma}
\begin{proof}
  Observe that $S_{n-1}$ has either smaller genus than $S_n$, or fewer
  boundary components than $S_{n-1}$ (as, otherwise, $S_{n-1}$ and
  $S_n$ would be isotopic). Now the claim follows by induction on genus
  and number of boundary components, lexicographically ordered.
\end{proof}

Recall that a set of curves $\calc=\{\alpha_i, i \in I\}$ on $\Sigma$
\emph{fills} a (non-annular) connected subsurface $X$ if no essential
curve $\beta$ in $X$ is disjoint from $\cup \alpha_i$ up to
homotopy. In the case where $X$ is an annulus we say that the core curve
fills $X$. 

We recall the following standard fact.
\begin{lemma}\label{lem:subsurface-filled}
  Suppose that $\calc$ is a set of curves on $\Sigma$. Then there is a
  subsurface $S_\calc \subset \Sigma$ containing $\calc$, so that $\calc$ is
  a filling set of curves on $S_\calc$. The subsurface $S_\calc$ is unique up to isotopy, and is called the \emph{subsurface filled by $\calc$}.
\end{lemma}
\begin{proof}
  Let $S_\calc$ be a subsurface containing $\calc$ which is minimal
  under inclusion with this property. The existence of such a
  subsurface follows from Lemma~\ref{lem:subsurface-chains}. This
  subsurface is filled by $\calc$ (as otherwise, the complement of a
  curve in $S_\calc$ disjoint from $\calc$ is a smaller subsurface).
  If $S'$ were a different subsurface filled by $\calc$, then
  $S'\cap S_\calc$ would be a smaller subsurface containing $\calc$, violating
  minimality. Hence, uniqueness follows.
\end{proof}

\begin{lemma}\label{lem:fixing-filling}
  Suppose that $\calc$ is a set of curves, which is preserved by a
  mapping class $f$. Assume either that
  \begin{itemize}
  \item $\calc$ is finite, or
  \item $f$ fixes every curve in $\calc$, i.e.\ $f(\alpha) = \alpha$ for every $\alpha \in \calc$.
  \end{itemize}
  Then the restriction $f\vert_{S_\calc}$ is finite order. In
  particular, if $f$ is contained in $\Mod^0(\Sigma)$, then $f$ is
  supported in the complement of $S_\calc$.
\end{lemma}
\begin{proof}
  Under either assumption, a power of $f$ fixes every curve in
  $\calc$. This implies that $f\vert_{S_\calc} \in \Mod(S_\calc)$ has
  finite order (see e.g.\ \cite[Proposition~2.8]{FM}). If
  $f\in\Mod^0(\Sigma)$ it is rotationless, and therefore the
  restriction $f\vert_{S_\calc}$ is also rotationless, and therefore
  the identity.
\end{proof}

\paragraph*{Stabilisers.}
\label{sec:bg-stab}

Central to our arguments is an understanding of stabilisers of curves
and subsurfaces in the mapping class group. Suppose that $Y$ is a
subsurface of $\Sigma$, and $[\varphi] \in \Mod(\Sigma)$ is a mapping
class which preserves the isotopy class of $Y$. One can choose a
representative homeomorphism $\varphi$ which fixes $Y$
setwise. Furthermore, the restriction of $\varphi$ to $Y$ is unique up
to isotopy.

In other words, we have a \emph{restriction homomorphism} from the
stabiliser of $Y$ to the mapping class group of $Y$,
\[ \mathrm{Stab}_{\Mod(\Sigma)}(Y) \to \Mod(Y). \]

In particular, let $X = \{ \alpha_1, \ldots, \alpha_k \}$ be a
multicurve in $\Sigma$, and let $\Sigma_X = \Sigma\setminus X$ be the
complement.  
In this setting, we then also have a \emph{restriction homomorphism} 
\[ \mathrm{Stab}_{\Mod(\Sigma)}(X) \to \Mod(\Sigma_X) \] to the
mapping class group of the (possibly disconnected) surface $\Sigma_X$.
The kernel of this restriction homomorphism is exactly the subgroup
generated by the Dehn twists about the curves in $X$ (compare
\cite[Proposition 3.20]{FM}), and so is in particular free
abelian.  One could describe the image precisely as well, but we will
not need this description.\footnote{For the curious reader: each curve
  in $X$ gives rise to a pair of punctures in $\Sigma_X$ corresponding
  to the side. The image consists of those mapping classes which
  respect these pairs.}

\smallskip Now suppose $Y \subset \Sigma_X$ is a connected
component. If $f$ is a rotationless mapping class preserving $X$, then
$f$ actually preserves $Y$ setwise. This is because by finiteness of
$X$, some power of $f$ fixes all $\alpha_i$ -- hence, by being
rotationless, $f$ already fixes each $\alpha_i$. Furthermore, since
the $\alpha_i$ are fixed as oriented curves (again, by definition of
rotationless), $f$ preserves each side of $\alpha_i$, hence all
complementary components. Thus, we have restriction homomorphisms
\[ \mathrm{Stab}_{\Mod^0(\Sigma)}(X) \to \Mod(Y) \] from the
stabiliser of a multicurve to the mapping class groups of each
complementary component.

\paragraph*{Canonical reduction multicurve, classical.}
Suppose that $H < \Mod(\Sigma)$ is a subgroup. A multicurve is called
a \emph{reduction multicurve for $H$} (sometimes also called a reduction system in the literature) if it is preserved by $H$ (up to
isotopy). We then say that $H$ is \emph{reducible} if it admits a non-empty reduction multicurve.
We will need the following theorem of Ivanov which can guarantee the existence
of pseudo-Anosov elements.

\begin{theo}[{Ivanov, \cite[Theorem~5.9]{Iva-book}}]\label{theo:Ivanov}
Let $\Sigma$ be a finite type surface of negative
  Euler characteristic which is not a pair of pants. Suppose that
  $H < \Mod^0(\Sigma)$ is an infinite subgroup. Either $H$ contains a
  pseudo-Anosov element, or $H$ is reducible.
\end{theo}
A reduction multicurve is \emph{maximal} if it is maximal under inclusion.
We define the \emph{canonical reduction multicurve} for $H$ to be the
intersection of all maximal reduction multicurves (compare \cite[Chapter~7]{Iva-book}). Observe that no curve
that intersects the canonical reduction multicurve is fixed by $H$ (as
otherwise there would be a maximal reduction multicurve which contains that
curve).

We mention that in the case of a subgroup $H\subseteq\Mod^0(\Sigma)$, Ivanov's theorem implies that the restriction of $H$ to every connected component of the complement of the canonical reduction multicurve of $H$ either contains a pseudo-Anosov element or is reduced to the identity. Therefore, the canonical reduction multicurve can also be obtained as follows. First define the \emph{canonical reduction set} $\calc_H$ of $H$ as the collection of all essential simple closed curves that are $H$-invariant, up to isotopy. Then the canonical reduction multicurve of $H$ is the union of the isolated curves in $\calc_H$ and of the boundaries of the subsurface $S_{\calc_H}$ filled by $\calc_H$. 

We also need the following well-known lemma.
\begin{lemma}\label{lem:commute-disjoint-crs}
  Suppose that $f,g$ are two commuting infinite order elements. Then
  the canonical reduction systems of $f$ and $g$ are disjoint up to isotopy.
\end{lemma}
\begin{proof}
  Since $f,g$ commute, $f$ sends every $g$-invariant curve to a
  $g$-invariant curve, and therefore $f$ preserves the canonical
  reduction multicurve of $g$. However, if $\alpha$ is a curve in the
  canonical reduction system of $f$, then $f$ preserves no curve
  $\beta$ which intersects $\alpha$. This shows the lemma.
\end{proof}


One can use Ivanov's theorem to easily generate pseudo-Anosov
elements (compare e.g. the discussion in \cite[Section~2.4]{Man}).
\begin{lemma}\label{lem:generating-pA}
  Suppose that $\alpha, \beta$ fill $\Sigma$. Let $T_\alpha, T_\beta$ be
  multitwists about $\alpha, \beta$ so that all twist powers are
  nonzero. Then the group $\langle T_\alpha, T_\beta \rangle$ contains a pseudo-Anosov.
\end{lemma}
\begin{proof}
  Observe that the only curves fixed by $T_\alpha$ are disjoint from
  $\alpha$ up to isotopy (compare
  e.g. \cite[Proposition~3.2]{FM}). Hence, there is no curve which
  is fixed by both $T_\alpha$ and $T_\beta$. Ivanov's theorem (Theorem~\ref{theo:Ivanov}) now
  gives the lemma.
\end{proof}

\subsection{Handlebody background}

\paragraph*{Handlebodies.} By a \emph{handlebody} of (finite) genus $g\ge 0$, we mean a connected orientable $3$-manifold which is a disk-sum of $g$ copies of $D^2\times S^1$, where $D^2$ is a closed disk and $S^1$ is a circle. The boundary $\partial V$ of a handlebody $V$ of genus $g$ is a closed, connected, orientable surface of the same genus $g$. The \emph{extended handlebody group} $\Mod^\pm(V)$ is the mapping class group of $V$, i.e.\ the group of all isotopy classes of homeomorphisms of $V$. The \emph{handlebody group} $\Mod(V)$ is the subgroup formed by orientation-preserving homeomorphisms. There is a restriction homomorphism $\Mod^\pm(V)\to\Mod^\pm(\partial V)$, which is injective, thus allowing us to view $\Mod^\pm(V)$ as a subgroup of $\Mod^\pm(\partial V)$ (see e.g.\ \cite[Lemma~3.1]{Hen-primer}).

\paragraph*{Curves, meridians and annuli.} 

Let  $V$ be a handlebody. 
An essential simple closed curve on $\partial V$ is a \emph{meridian} (represented in blue in Figure~\ref{fig:meridian}) if it bounds a properly embedded disk in $V$. 

\begin{figure}
	\begin{center}
		\includegraphics[width=0.7\textwidth]{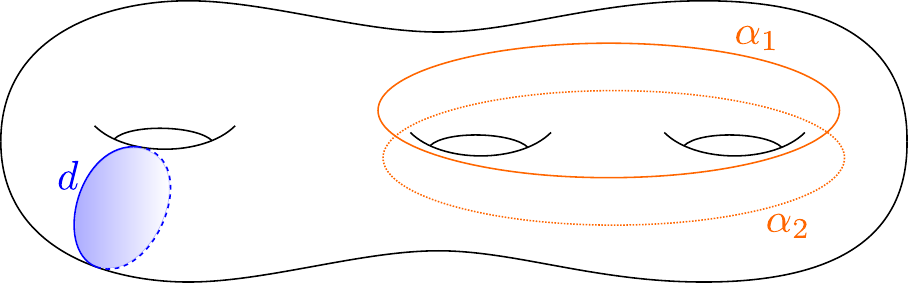}
	\end{center}
	\caption{On the left: a \emph{meridian} $d$, i.e.\ an essential curve bounding a disk in the handlebody.
		On the right: two curves $\alpha_1, \alpha_2$ which individually do not bound disks in the handlebody, and which are not homotopic on the boundary surface, but bound a properly embedded annulus in the handlebody.}
	\label{fig:meridian}
\end{figure}

If $c\subseteq\partial V$ is a meridian, then the Dehn twist $T_c$
associated to $c$ belongs to $\Mod(V)$, viewed as a subgroup of
$\Mod(\partial V)$ -- and this is in fact a characterisation of
meridians, as follows from \cite[Theorem~1]{McC} or
\cite[Theorem~1.11]{Oer}.

For multitwists, there is another possibility. Namely, a
pair $\{\alpha_1,\alpha_2\}$ of disjoint nonisotopic essential simple
closed curves on $\partial V$ is an \emph{annulus pair} (represented
in red in Figure~\ref{fig:meridian}) if neither $\alpha_1$ nor
$\alpha_2$ is a meridian, and there exists a properly embedded annulus
$A\subseteq V$ such that $\partial A= \alpha_1\cup \alpha_2$. An \emph{annulus twist} is a mapping class of the form
$T_{\alpha_1}T_{\alpha_2}^{-1}$ for some annulus pair
$\{\alpha_1,\alpha_2\}$. Annulus twists belong to $\Mod(V)$ (\cite[Theorem~1]{McC} or
\cite[Theorem~1.11]{Oer}).

\begin{lemma}\label{lemma:pA-in-complement}
Let $c$ be a meridian. Then every connected component of $\partial V\setminus c$ which is not a once-holed torus supports two handlebody group elements which both restrict to a pseudo-Anosov mapping class of $\partial V\setminus c$ and together generate a nonabelian free subgroup.
\end{lemma}

\begin{proof}
  Let $X$ be a connected component of $\partial V\setminus c$ which is
  not a once-holed torus, and denote by $c_1$ a boundary component of $X$ 
  (corresponding to one of the sides of $c$). Fix an essential simple closed curve
  $\alpha \subset X$ which is not boundary parallel in $X$. We can (and shall) choose an essential simple closed curve $\alpha' \subset X$ which is not isotopic to $\alpha$ and  
  bounds a pair of pants on $X$ together with $c_1, \alpha$ (here, we
  are using that $X$ is not a once-holed torus). Since $c$ is a
  meridian, $\alpha, \alpha'$ are either both meridians, or form an
  annulus pair. Thus, in either case, the multitwist
  $f_\alpha = T_{\alpha'}T_\alpha^{-1}$ is a handlebody group element
  supported in $X$.

  By choosing curves $\alpha, \beta$ which fill $X$
  the group generated by $f_\alpha, f_\beta$ contains a pseudo-Anosov $\psi$ by Lemma~\ref{lem:generating-pA}. Conjugating $\psi$ by $f_\alpha$ yields a
  second one, and sufficiently high powers of $\psi$ and $f_{\alpha}\psi f_{\alpha}^{-1}$ generate a nonabelian free subgroup.
\end{proof}

Components of $\partial V\setminus c$ which are once-holed tori behave differently, as shown by the following lemma.

\begin{lem}\label{lem:meridian-stab-containment}
  Suppose that $c$ is a separating meridian, and suppose that $X$ is a
  component of $\partial V \setminus c$ which is a
  once-holed torus. Then $X$ contains a unique (nonseparating) meridian $d_X$ which is not peripheral in $X$ up to
  isotopy, and therefore  \[ \mathrm{Stab}_{\mathrm{Mod}(V)}(c) \subsetneq
    \mathrm{Stab}_{\mathrm{Mod}(V)}(d_X). \] If the genus of $V$ is at
  least $3$, then $d_X$ is the only other meridian whose stabiliser
  contains $\mathrm{Stab}_{\mathrm{Mod}(V)}(c)$ (or even a
  finite-index subgroup of $\Stab_{\Mod(V)}(c)$).
\end{lem}

\begin{proof}
	The subsurface $X$ is the boundary of a once-spotted
	genus $1$ handlebody $V_1^1$. Hence, there is a nonseparating meridian $d_X$ contained
	in $X$. We claim that it is the only one up to isotopy. Namely, recall that in a once-holed torus
	any two isotopically distinct essential simple closed curves have nonzero algebraic intersection number (compare e.g. \cite[Lemma~2.1]{Hen-primer}). However, any two meridians have algebraic intersection number zero.
	
In particular $\Stab_{\Mod(V)}(c)\subseteq\Stab_{\Mod(V)}(d_X)$. This inclusion is strict: indeed, by Lemma~\ref{lemma:pA-in-complement}, there exists a handlebody group element $\varphi$ which fixes $d_X$ and restricts to a pseudo-Anosov homeomorphism on the complementary subsurface, in particular $\varphi$ does not fix the isotopy class of $c$.	
	
	To show the final claim, recall from Lemma~\ref{lemma:pA-in-complement} that there are elements in $\mathrm{Stab}_{\mathrm{Mod}(V)}(c)$ restricting to pseudo-Anosov elements on any component of $\partial V \setminus c$ which is not a once-holed torus. If the genus of $V$ is at least $3$, the complement of $X$ will be such a component. Hence, $d_X$ is the unique other meridian fixed by $\mathrm{Stab}_{\mathrm{Mod}(V)}(c)$ (or any finite-index subgroup).
\end{proof}

For the next corollary and below, we put
\[ \Mod^0(V) = \Mod(V) \cap \Mod^0(\partial V). \]
\begin{cor}\label{cor:once-holed}
  Let $c$ be a separating meridian. Let $\Sigma_1,\Sigma_2$ be the two components of $\partial V\setminus c$, and suppose that $\Sigma_2$ is a
  once-holed torus. 
  
  Then the kernel of the restriction homomorphism $\Stab_{\Mod^0(V)}(c)\to\Mod(\Sigma_1)$ is isomorphic to $\mathbb{Z}^2$. 
\end{cor}
  
\begin{proof}
  Suppose $f$ lies in the kernel of the restriction homomorphism
  $\Stab_{\Mod^0(V)}(c)\to\Mod(\Sigma_1)$. Then $f$ is supported on
  $\Sigma_1$, and the restriction $f\vert_{\Sigma_1}$ is rotationless.

  Furthermore, let $d \subset \Sigma_1$ be the unique meridian, which
  is also fixed by $f$ (compare
  Lemma~\ref{lem:meridian-stab-containment}). Since $\Sigma_1$ is a
  once-holed torus, the complement of $d$ in $\Sigma_1$ is a
  three-holed sphere. Since a mapping class of a three-holed sphere
  which preserves every component is trivial \cite[Proposition~2.3]{FM}, the description of
  stabilisers implies that $f$ is a product of the twist about $c$ and
  the twist about $d$. This shows the corollary.
\end{proof}



\begin{cor}\label{cor:sep-meridian}
Let $c$ be a separating meridian. Then there exists $g\in\Mod(V)$ such that for every $n\neq 0$, the curve $c$ is the only separating meridian whose isotopy class is fixed by $g^n$.
\end{cor}

\begin{proof}
  Let $S$ be the union of all complementary components of $c$ which
  are not once-holed tori. Observe that $c$ is the only
  separating curve which does not intersect $S$ (since a once-holed
  torus contains no essential separating curve).

  Hence, if we let $g$ be a handlebody group element
  which restricts to a pseudo-Anosov on each component of $S$ (which is possible
  by Lemma~\ref{lemma:pA-in-complement}), it has the desired property.  
\end{proof}

\subsection{Improved rotationless mapping classes}\label{sec:rotationless} 

In this section we discuss the following technical issue. For
inductive arguments with cocycles, it is convenient to consider finite
index subgroups of the mapping class group which consist only of
rotationless elements.  As discussed above, every mapping class $f$ in
$\Mod^0(\Sigma)$ is rotationless. If $S \subset \Sigma$ is a
subsurface which is preserved by $f$, then the restriction $f\vert_S$
is again rotationless by definition -- however, the restriction
$f\vert_S$ need not be contained in $\Mod^0(S)$. Namely: curves
contained in $S$, which are homologous when seen as curves $\Sigma$,
need not be homologous inside $S$.

To avoid this issue, we will prove in this section the following.
\begin{lemma}\label{lem:Mod1}
  Let $\Sigma$ be a connected surface. There is a finite index subgroup $\Mod^1(\Sigma)$ of
  $\Mod^0(\Sigma)$ such that if $h \in \Mod^1(\Sigma)$ is any element
  preserving a connected subsurface $S \subset \Sigma$, then the
  restriction $h\vert_S$ is an element of $\Mod^0(S)$.
\end{lemma}

\begin{rk}
Strictly speaking, it would be possible to avoid Lemma~\ref{lem:Mod1} and take a slightly different route in Section~\ref{sec:me}, see Remark~\ref{rk:rotationless-cocycle} there. A reader unfamiliar with homology arguments may therefore prefer to skip this part for now. On the other hand, we believe that from the geometric viewpoint, working with the finite-index subgroup $\Mod^1(V)$ is more natural, so we decided to include the present lemma. 
\end{rk}

Our proof of Lemma~\ref{lem:Mod1} constructs a specific finite-index
subgroup $\Mod^1(\Sigma)$ using covering spaces, and in the sequel of
the paper, we will work with this subgroup throughout. This
construction may be known to experts (but we were unable to
locate it in the literature).  Namely, denote by $p:X \to \Sigma$ the
\emph{mod-$2$-homology cover} of the surface $\Sigma$, which is the
cover defined by the surjection
\[ \pi_1(\Sigma, b_0) \to H_1(\Sigma; \mathbb{Z}/2) \] of the
fundamental group (for any basepoint $b_0$) to homology with mod-$2$--coefficients. Note that the
mod-$2$-homology cover $p:X\to\Sigma$ is \emph{characteristic}: every
homeomorphism of $\Sigma$ lifts to a homeomorphism of
$X$ (since the action on integral homology
preserves being divisible by $2$). Also note that if
$F:\Sigma \to \Sigma$ is a homeomorphism, then a lift
$\widetilde{F}:X \to X$ is well-defined up to the action on the deck
transformation group of the cover $X$.
\begin{de}\label{de:mod1}
  Let $\Sigma$ be a connected surface. The subgroup $\Mod^1(\Sigma)$
  consists of those mapping classes $[F]$ of $\Sigma$ which admit a
  lift $\widetilde{F}$ to $X$ that acts trivially on
  $H_1(X;\mathbb{Z}/3\mathbb{Z})$.

  \smallskip We put
  \[ \Mod^1(V)=\Mod(V)\cap\Mod^1(\partial V) \]
\end{de}

The advantage we gain by using the mod-$2$-homology cover is that lifts of subsurfaces
inject in homology, avoiding the problem mentioned
above:
\begin{lemma}\label{lem:cover-homology-inclusion}
  Let $S \subset \Sigma$ be an essential connected subsurface. Denote
  by $p:X \to \Sigma$ the mod-$2$-homology cover, and let
  $X_S \subset X$ be a connected component of $p^{-1}(S)$. Then the
  map
  \[ H_1(X_S; \mathbb{Z}) \to H_1(X; \mathbb{Z}) \] induced by the
  inclusion is injective, and the same is true with $\mathbb{Z}$
  replaced with $\mathbb{Z}/n$ for any $n$.
\end{lemma}
\begin{proof}
  Observe that there is nothing to show if $X_S$ has a single boundary
  component -- such surfaces always induce inclusions in
  homology. Hence we may assume that $X_S$ has at least two boundary
  components throughout. 

  To prove the lemma we will construct a basis of homology of $X_S$
  which is linearly independent in the homology of $X$. The reasons
  for independence will be geometric, and hence work with any
  coefficients. Thus, we restrict to the integral case for ease of
  notation.

  \begin{figure}[h!]
    \centering
    \includegraphics[width=0.8\textwidth]{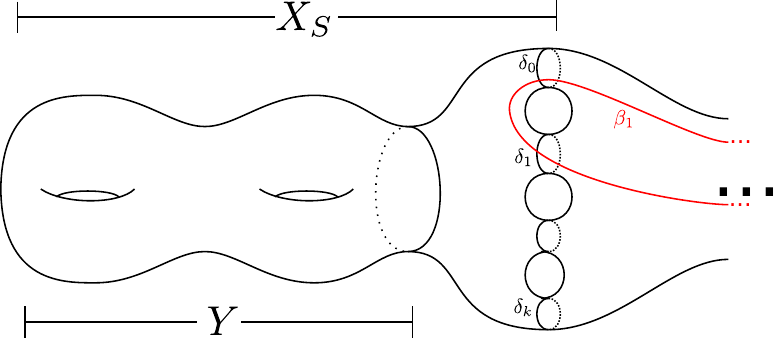}
    \caption{The setup in the proof of
      Lemma~\ref{lem:cover-homology-inclusion}. The subsurface $X_S$
      is decomposed into a surface $Y$ with a single boundary, and a
      bordered sphere. To control the homology classes defined by the
      boundary curves of that sphere, we construct auxiliary curves $\beta_i$.}
    \label{fig:hombasis}
  \end{figure}
  
  \smallskip First, we choose a subsurface $Y \subset X_S$ with one
  boundary component, so that $X_S \setminus Y$ is a bordered
  sphere. Then $H_1(Y;\mathbb{Z})$ injects into $H_1(X;\mathbb{Z})$,
  since $Y$ is a subsurface of $X$ with one boundary component. If we
  denote the boundary curves of $X_S$ by $\delta_0, \ldots, \delta_k$,
  then the homology of the bordered sphere $X_S \setminus Y$ is
  generated by the $[\delta_i]$ and in fact we have
  \[ H_1(X_S;\mathbb{Z}) = H_1(Y;\mathbb{Z}) \oplus \mathbb{Z}^k, \]
  where the latter summand is generated by
  $[\delta_1], \ldots, [\delta_k]$.

  \smallskip We now aim to show that for all $i > 0$ there is a curve
  $\beta_i$ which is disjoint from $Y$, intersects
  $\delta_0, \delta_i$ each in a single point, and is disjoint from
  all other $\delta_j$. This will show that
  $[\delta_1], \ldots, [\delta_k]$ are linearly independent from each
  other and from $H_1(Y;\mathbb{Z})$ in $H_1(X;\mathbb{Z})$ thus
  showing the lemma.
	
  \smallskip For simplicity of notation, we will perform the
  construction only for $i=1$. Choose a basepoint $\widetilde{q}$ in
  $Y$, and let $q = p(\widetilde{q})$ be its image in $\Sigma$. Since
  the mod-$2$ homology cover is normal (Galois), the preimage
  $p^{-1}(q)$ is exactly the orbit of $\widetilde{q}$ under the deck
  group $D = H_1(\Sigma;\mathbb{Z}/2)$.

  To describe the intersection $p^{-1}(q) \cap X_S$, first observe
  that since $X_S$ is connected, a point $\widetilde{q}'\in p^{-1}(q)$
  is contained in $X_S$ exactly if there is a path
  $\widetilde{\gamma}$ connecting $\widetilde{q}$ to $\widetilde{q}'$
  contained in $X_S$. Such paths are exactly the lifts of loops
  $\gamma$ based at $q$ which are contained in $S$. So
  $\widetilde{q}'$ is contained in $X_S$ if and only if the deck group
  element $g$ mapping $\widetilde{q}$ to $\widetilde{q}'$ is the image
  of some $\gamma \in \pi_1(S, q) \subseteq \pi_1(\Sigma,q)$. The
  image of $\pi_1(S, q)$ in the deck group is exactly the subgroup
  $D_S = \mathrm{im}(H_1(S;\mathbb{Z}/2) \to
  H_1(\Sigma;\mathbb{Z}/2))$. Together this shows that
  $p^{-1}(q) \cap X_S = D_S\widetilde{q}$.
        
	Similarly, the components of $p^{-1}(S)$ can be identified with the	cosets of the subgroup $D_S \subseteq D$.
	
	To describe the cover more precisely, we choose curves $\gamma_i$ based at $q$
	in the following way:
	\begin{enumerate}[(1)]
        \item The homology classes $[\gamma_i] = x_i$ form a basis
          $x_1, \ldots, x_N$ of $H_1(\Sigma;\mathbb{Z})$,
        \item $x_1, \ldots, x_k$ is a basis of
          $\mathrm{im}(H_1(S;\mathbb{Z}) \to H_1(\Sigma;\mathbb{Z}))$,
          and the curves $\gamma_i$ are contained in $S$.
        \item The curves $\gamma_i$ for $i = k+1, \ldots, N$ intersect
          $\partial S$ in exactly two points. 
        \end{enumerate}
        To see that these curves exist, we argue as follows. Denote by
        $S_1, \ldots, S_r$ the components of $\Sigma \setminus S$.  Choose a
        curve $\alpha_i \subset \partial S_i$. The connectivity of $S$ implies that for every $i\in\{1,\dots,r\}$, the curve $\alpha_i$
        is homologically nontrivial (in $H_1(\Sigma)$) exactly if
        $\partial S_i$ has more than one component.  For each boundary curve
        $\beta \subset \partial S_i \setminus \alpha_i$ we can
        find a loop $\gamma_\beta$ based at $q$ which intersects
        $\partial S$ in two points, one on $\beta$ and one on
        $\alpha_i$. We can thus choose independent homology classes
        $z_i$ defined by curves intersecting $\partial S$ in at most
        two points, so that for any $x \in H_1(\Sigma)$ there is a
        linear combination $z$ of the $z_i$, so that $x+z$ has
        algebraic intersection number $0$ with all curves in
        $\partial S$. Any such class $x+z$ can be realised by a
        multicurve disjoint from $\partial S$. Since every homology
        class defined by a curve (without specified basepoint) in $S_i$ can be realised by a loop
        based at $q$ which intersects $\partial S$ in two
        points, and every curve in $S$ can be realised by a loop disjoint from $\partial S$ the desired existence follows. 

        Lifting a curve of the type in (2) at a point
        $h\widetilde{q}$ stays in the same connected component $hX_S$,
        while lifting a curve of the type in (3) joins $hX_S$ to
        $h'X_S$ and intersects $\partial hX_S$ in a single point. To
        see that last claim observe that a lift of a curve as in (3)
        cannot join two points of $hX_S$, as the image of that curve
        in $ H_1(\Sigma;\mathbb{Z}/2)$ would then be contained in
        $\mathrm{im}(H_1(S;\mathbb{Z}/2) \to
        H_1(\Sigma;\mathbb{Z}/2))$, contradicting (1) and (2).
	
	For every $i\in\{0,1\}$, denote by $Z_i$ the
        component of $X \setminus p^{-1}(S)$ adjacent to $\delta_i$.
        There are
        $h_i \notin \mathrm{im}(H_1(S;\mathbb{Z}/2) \to
        H_1(\Sigma;\mathbb{Z}/2))$, so that $h_iX_S$ are surfaces
        adjacent to $Z_i$.  Namely, either $p(Z_i)$ has genus (which is
        automatically the case if $p(\delta_i)$ is separating), and
        contains a curve defining one of the $x_j$ (as in (3)), or $p(Z_i)$ is a punctured sphere so that for the
        boundary component $p(\delta_i)$ there is some $x_j, j>k$ (of
        the third type) which intersects it once (namely, if all $x_i$
        would interesect $p(\delta_i)$ in an even number of points,
        the $x_i$ could not be a basis of $H_1(\Sigma;\mathbb{Z})$,
        since $p(\delta_i)$ is nonseparating).  In both cases the
        desired component is $\pm[x_j]X_S$.  Choose paths $c_i, i=0,1$
        joining $\widetilde{q}$ to $h_i\widetilde{q}$ which intersect
        only $\delta_i$ among the $\delta_j$.
	
	Since
	$\mathrm{im}(H_1(S;\mathbb{Z}/2) \to H_1(\Sigma;\mathbb{Z}/2)) =
	(\mathbb{Z}/2)^k$ is a subgroup of
	$H_1(\Sigma;\mathbb{Z}/2)) = (\mathbb{Z}/2)^N$ generated by a subset
	of the generators, there is a path in the Cayley graph of
	$H_1(\Sigma;\mathbb{Z}/2)$ from $h_0$ to $h_1$ which is disjoint
	from the Cayley graph of the subgroup
	$\mathrm{im}(H_1(S;\mathbb{Z}/2) \to H_1(\Sigma;\mathbb{Z}/2)$. Each edge in such a path
	corresponds to a right multiplication $h \mapsto hx_s$, and we can choose a corresponding
	path joining $h\widetilde{q}$ to $hx_s\widetilde{q}$ which is disjoint from $X_S$.
	By concatenating these paths with $c_0, c_1$ (in the right order) we
	then find the desired path $\beta_1$.
\end{proof}

\begin{cor}\label{cor:really-good-subgroup}
  Suppose that $F$ is a homeomorphism so that
  \begin{enumerate}
  \item $F$ admits a lift $\widetilde{F}$ to the
    mod-$2$-homology-cover $X$, which acts trivially on
    $H_1(X;\mathbb{Z}/3)$
  \item $F$ preserves a subsurface $S \subset \Sigma$
  \end{enumerate}
  Then the restriction $F\vert_S$ acts trivially on
  $H_1(S;\mathbb{Z}/3)$.
\end{cor}
\begin{proof}
  Let $\alpha \subset S$ be a simple closed curve which is part of a
  basis for $H_1(S;\mathbb{Z}/3)$.  Then there is a power $N = 2^n$ so
  that $\alpha^N$ lifts to a curve $\widetilde{\alpha} \subset X_S$
  (with notation as in the previous lemma), since $p$ is a
  finite-sheeted cover.
	
  Denote by $p_S: X_S \to S$ the restriction of the covering map
  (which is then also a covering). We have
  $(p_S)_\ast[\widetilde{\alpha}] = N[\alpha]$. Since $N$ is
  invertible mod $3$, there is a multiple $k$ so that
  $(p_S)_\ast k[\widetilde{\alpha}] = [\alpha]$ mod $3$.
	
  By Lemma~\ref{lem:cover-homology-inclusion}, the inclusion of $H_1(X_S;\mathbb{Z}/3)$
  into $H_1(X;\mathbb{Z}/3)$ is injective. Since $\widetilde{F}$ acts
  trivially on $H_1(X;\mathbb{Z}/3)$, this implies that the
  restriction $\widetilde{F}_{X_S}$ acts trivially on
  $H_1(X_S;\mathbb{Z}/3)$.  Hence, we have
  $(\widetilde{F}_{X_S})_\ast k[\widetilde{\alpha}] =
  k[\widetilde{\alpha}]$. Since $\widetilde{F}_{X_S}$ is a lift of
  $F_S$ this implies $(F_S)_\ast[\alpha] = [\alpha]$.
\end{proof}

The central Lemma~\ref{lem:Mod1} is now an immediate consequence of
Corollary~\ref{cor:really-good-subgroup} and the fact that there are
only finitely many automorphisms of $H_1(X;\mathbb{Z}/3\mathbb{Z})$.

\subsection{Infinite conjugacy classes}

 A countable group $G$ is said to be \emph{ICC} (standing for \emph{infinite conjugacy classes}) if the conjugacy class of every nontrivial element of $G$ is infinite.

\begin{lemma}\label{lemma:icc}
  Let $V$ be a handlebody of genus at least $2$, and let $\varphi \in \Mod^{\pm}(V)$ be a handlebody group element. Then either the conjugacy class of $\varphi$ is infinite,
  or $\varphi$ fixes the isotopy class of every meridian.
  
  In particular, when the genus of $V$ is at least $3$, the group $\Mod^{\pm}(V)$ is ICC.
\end{lemma}

We remark that in genus $2$, 
the hyperelliptic involution fixes the isotopy class of every essential simple closed curve on $\partial V$, and its conjugacy class is finite in $\Mod^{\pm}(V)$. 

\begin{proof}
Suppose that $\varphi$ is an element with finite conjugacy class. For any meridian $c$,
consider the elements $T_c^i\varphi T_c^{-i}$ for $i \in \mathbb{N}$. By finiteness of
the conjugacy class, two of these have to be equal, and thus there is some $N>0$ so that
\[ T_c^N\varphi T_c^{-N} = \varphi, \]
or equivalently,
\[ T_c^N = \varphi T_c^{N}\varphi^{-1} = T_{\varphi(c)}^N. \]
This implies $c$ is isotopic to $\varphi(c)$, see e.g.\ \cite[Section~3.3]{FM} (which also holds for orientation-reversing mapping classes). The first part of the lemma follows since $c$ was arbitrary. The fact that $\Mod^{\pm}(V)$ is ICC when the genus is at least $3$ follows because every element fixing the isotopy class of every meridian is then trivial \cite[Theorem~9.4]{KS}.
\end{proof}

\section{Background on measured groupoids}

The reader is refered to \cite[Section~2.1]{AD}, \cite{Kid-survey} or \cite[Section~3]{GH} for general background on measured groupoids. 

Recall that a \emph{standard Borel space} is a measurable space associated to a Polish space (i.e.\ separable and completely metrizable). A \emph{standard probability space} is a standard Borel space equipped with a Borel probability measure.

A \emph{Borel groupoid} is a standard Borel space $\calg$ (whose elements are thought of as being arrows) equipped with two Borel maps $s,r:\calg\to Y$ towards a standard Borel space $Y$ (giving the source and range of an arrow), and coming with a measurable (partially defined) composition law and inverse map and with a unit element $e_y$ per $y\in Y$. The Borel space $Y$ is called the \emph{base space} of the groupoid $\calg$. All Borel groupoids considered in the present paper are assumed to be \emph{discrete}, i.e.\ there are countably many arrows in $\calg$ with a given range (or source). It follows from a theorem of Lusin and Novikov (see e.g.\ \cite[Theorem~18.10]{Kec}) that a discrete Borel groupoid $\calg$ can always be written as a countable disjoint union of \emph{bisections}, i.e.\ Borel subsets $B$ of $\calg$ on which $s$ and $r$ are injective (in which case $s(B)$ and $r(B)$ are Borel subsets of $Y$, see \cite[Corollary~15.2]{Kec}). A Borel groupoid $\calg$ with base space $Y$ is \emph{trivial} if $\calg=\{e_y|y\in Y\}$.

A finite Borel measure $\mu$ on $Y$ is \emph{quasi-invariant} for the groupoid $\calg$ if for every bisection $B\subseteq\calg$, one has $\mu(s(B))>0$ if and only if $\mu(r(B))>0$. A \emph{measured groupoid} is a Borel groupoid together with a quasi-invariant finite Borel measure on its base space $Y$. 

An important example of a measured groupoid to keep in mind is the following: when a countable group $G$ acts on a standard probability space $Y$ by Borel automorphisms in a quasi-measure-preserving way, then $G\times Y$ has a natural structure of a measured groupoid over $Y$, denoted by $G\ltimes Y$: the source and range maps are given by $s(g,y)=y$ and $r(g,y)=gy$, the composition law is $(g,hy)(h,y)=(gh,y)$, the inverse of $(g,y)$ is $(g^{-1},gy)$ and the units are $e_y=(e,y)$.

A Borel subset $\calh\subseteq\calg$ which is stable under composition and inverse and contains all unit elements of $\calg$ has the structure of a \emph{measured subgroupoid} of $\calg$ over the same base space $Y$. Given a Borel subset $U\subseteq Y$, the \emph{restriction} $\calg_{|U}$ is the measured groupoid over $U$ defined by only keeping the arrows whose source and range both belong to $U$. Given two subgroupoids $\calh,\calh'\subseteq\calg$, we denote by $\langle\calh,\calh'\rangle$ the subgroupoid of $\calg$ generated by $\calh$ and $\calh'$, i.e.\ the smallest subgroupoid of $\calg$ containing $\calh$ and $\calh'$ (it is made of all arrows obtained as finite compositions of arrows in $\calh$ and arrows in $\calh'$).

A measured groupoid $\calg$ with base space $Y$ is \emph{of infinite type} if for every positive measure Borel subset $U\subseteq Y$ and almost every $y\in U$, there are infinitely many elements of $\calg_{|U}$ with source $y$. Observe that if $\calg$ is of infinite type, then for every Borel subset $U\subseteq Y$ of positive measure, the restricted groupoid $\calg_{|U}$ is again of infinite type. 

Let $\calg$ be a measured groupoid over a standard probability space $Y$, and let $G$ be a countable group. A \emph{strict cocycle} $\rho:\calg\to G$ is a Borel map such that for all $g_1,g_2\in\calg$, if the source of $g_1$ is equal to the range of $g_2$ (so that the product $g_1g_2$ is well-defined), then $\rho(g_1g_2)=\rho(g_1)\rho(g_2)$. The \emph{kernel} of a cocycle $\rho$ is the subgroupoid of $\calg$ made of all $g\in\calg$ such that $\rho(g)=1$. We say that $\rho$ has \emph{trivial} kernel if its kernel is equal to the trivial subgroupoid of $\calg$, i.e.\ it only consists of the unit elements of $\calg$. We say that a strict cocycle $\calg\to G$ is \emph{action-type} if $\rho$ has trivial kernel, and whenever $H\subseteq G$ is an infinite subgroup, and $U\subseteq Y$ is a Borel subset of positive measure, then $\rho^{-1}(H)_{|U}$ is a subgroupoid of $\calg_{|U}$ of infinite type. Note that if $\rho:\calg\to G$ is an action-type cocycle, then for every positive measure Borel subset $U\subseteq Y$, the restriction $\rho:\calg_{|U}\to G$ is again action-type. An important example is that given a measure-preserving $G$-action on a standard probability space $Y$, the natural cocycle $\rho:G\ltimes Y\to G$ is action-type \cite[Proposition~2.26]{Kid-survey}. We warn the reader that in the latter example, it is important that the $G$-action on $Y$ preserves the measure, as opposed to only quasi-preserving it. 

Given a Polish space $\Delta$ equipped with a $G$-action by Borel automorphisms, we say that a measurable map $\varphi:Y\to\Delta$ is \emph{$(\calg,\rho)$-equivariant} if there exists a conull Borel subset $Y^*\subseteq Y$ such that for every $g\in\calg_{|Y^*}$, one has $\varphi(r(g))=\rho(g)\varphi(s(g))$. We say that an element $\delta\in\Delta$ is \emph{$(\calg,\rho)$-invariant} if the constant map with value $\delta$ is $(\calg,\rho)$-equivariant (equivalently, there exists a conull Borel subset $Y^*\subseteq Y$ such that $\rho(\calg_{|Y})\subseteq\Stab_G(\delta)$). The \emph{$(\calg,\rho)$-stabilizer} of $\delta$ is the subgroupoid of $\calg$ made of all elements $g$ such that $\rho(g)\in\Stab_G(\delta)$. A measurable map $\varphi:Y\to\Delta$ is \emph{stably $(\calg,\rho)$-equivariant} if one can partition $Y$ into at most countably many Borel subsets $Y_i$ such that for every $i$, the map $\varphi_{|Y_i}$ is $(\calg_{|Y_i},\rho)$-equivariant.

Given two measured subgroupoids $\calh,\calh'\subseteq\calg$, we say that $\calh$ is \emph{stably contained} in $\calh'$ if there exist a conull Borel subset $Y^*\subseteq Y$ and a partition $Y^*=\dunion_{i\in I}Y_i$ into at most countably many Borel subsets such that for every $i\in I$, one has $\calh_{|Y_i}\subseteq\calh'_{|Y_i}$. We say that $\calh$ and $\calh'$ are \emph{stably equal} if there exist a conull Borel subset and a partition as above such that for every $i\in I$, one has $\calh_{|Y_i}=\calh'_{|Y_i}$. We say that $\calh$ is \emph{stably trivial} if it is stably equal to the trivial subgroupoid of $\calg$.

Let $\calh$ be a measured subgroupoid of $\calg$, and $B\subseteq\calg$ be a bisection. We say that $\calh$ is \emph{$B$-invariant} if there exists a conull Borel subset $Y^*\subseteq Y$ such that for every $g_1,g_2\in B\cap\calg_{|Y^*}$ and every $h\in\calg_{|Y^*}$ such that the composition $g_2hg_1^{-1}$ is well-defined, we have $h\in\calh_{|Y^*}$ if and only if $g_2hg_1^{-1}\in\calh_{|Y^*}$. Let now $\calh'$ be another measured subgroupoid of $\calg$. The groupoid $\calh$ is \emph{normalized} by $\calh'$ if $\calh'$ can be covered by countably many bisections $B_n\subseteq\calg$ in such a way that $\calh$ is $B_n$-invariant for every $n\in\mathbb{N}$. The subgroupoid $\calh$ is \emph{stably normalized} by $\calh'$ if one can partition $Y$ into at most countably many Borel subsets $Y_i$ in such a way that for every $i$, the groupoid $\calh_{|Y_i}$ is normalized by $\calh'_{|Y_i}$. When $\calh\subseteq\calh'$, we will simply say that $\calh$ is \emph{stably normal} in $\calh'$.

There is a notion of \emph{amenability} of a measured groupoid, generalizing Zimmer's notion of amenability of a  group action, for which we refer to \cite{Kid-survey}; here we only list the properties of amenable groupoids we will need. First, if $\calg$ is amenable and comes with a cocycle $\rho:\calg\to G$ towards a countable group $G$, and if $G$ acts by homeomorphisms on a compact metrizable space $K$, then there exists a $(\calg,\rho)$-equivariant Borel map $Y\to\Prob(K)$, see \cite[Proposition~4.14]{Kid-survey}. Here $\Prob(K)$ denotes the set of Borel probability measures on $K$, equipped with the weak-$\ast$ topology coming from the duality with the space of real-valued continuous functions on $K$ given by the Riesz--Markov--Kakutani theorem. Second, whenever $\rho:\calg\to G$ is a cocycle with trivial kernel, and $A\subseteq G$ is an amenable subgroup of $G$, then $\rho^{-1}(A)$ is an amenable subgroupoid of $\calg$ (see e.g.\ \cite[Corollary~3.39]{GH}). Amenability is stable under subgroupoids and restrictions. Furthermore, if there exists a conull Borel subset $Y^*\subseteq Y$ and a partition $Y^*=\dunion_{i\in I}Y_i$ into at most countably many Borel subsets such that for every $i\in I$, the groupoid $\calg_{|Y_i}$ is amenable, then $\calg$ is amenable (this is immediate with the definition of amenability given in \cite[Definition~3.33]{GH}, see also \cite[Remark~3.34]{GH} for the comparison to equivalent definitions).

A groupoid $\calg$ over a standard probability space $Y$ is \emph{everywhere nonamenable} if for every Borel subset $U\subseteq Y$ of positive measure, the groupoid $\calg_{|U}$ is nonamenable.

Let us finish this section with the following lemma that we will use several times in the sequel.

\begin{lemma}\label{lemma:zorn}
Let $(X,\mu)$ be a standard probability space, and let $\mathfrak{F}$ be a set of Borel subsets of $X$ which is closed under countable unions.

Then $\mathfrak{F}$ has a maximal element $U$, i.e.\ such that every $V\in\mathfrak{F}$ has a conull Borel subset contained in $U$.
\end{lemma}

\begin{proof}
We claim that $\mathfrak{F}$ contains a subset $U$ of maximal measure. Indeed, if $(U_n)_{n\in\mathbb{N}}$ is a measure-maximizing sequence of subsets in $\mathfrak{F}$, then the union $U$ of the subsets $U_n$ belongs to $\mathfrak{F}$ by assumption, and has maximal measure.

Let now $U$ be as above, and let $V\in\mathfrak{F}$. If $V\setminus U$ were not a null set, then the measure of $U\cup V$ would be strictly larger than that of $U$, a contradiction. So $V$ has a conull Borel subset contained in $U$, as desired.
\end{proof}

\section{Measure equivalence rigidity of the handlebody group}\label{sec:me}

In this section, we prove the main theorem of the present paper. Throughout the section $V$ will always be a handlebody of genus at least $3$.

\begin{theo}\label{theo:main}
Let $V$ be a handlebody of genus at least $3$. Then $\Mod(V)$ is ME-superrigid.
\end{theo}

Recall that $\Mod^{\pm}(V)$ is ICC (Lemma~\ref{lemma:icc}), and that $\Mod^1(V)$ is the finite-index subgroup of $\Mod^{\pm}(V)$ introduced in Definition~\ref{de:mod1}. Thus, Theorem~\ref{theo:main} is a consequence of the following statement, combined with \cite[Theorem~4.5]{GH} (which builds on earlier works of Furman \cite{Fur2,Fur3} and Kida \cite{Kid}).

\begin{theo}\label{theo:main-2}
Let $V$ be a handlebody of genus at least $3$. Let $\calg$ be a measured groupoid over a standard probability space $Y$ (with source map $s$ and range map $r$), and let $\rho_1,\rho_2:\calg\to\Mod^1(V)$ be two strict action-type cocycles. 

Then there exist a Borel map $\theta:Y\to\Mod^{\pm}(V)$ and a conull Borel subset $Y^*\subseteq Y$ such that for all $g\in\calg_{|Y^*}$, one has $\rho_1(g)=\theta(r(g))^{-1}\rho_2(g)\theta(s(g))$.
\end{theo}

\begin{rk}\label{rk:commensurator}
The case where $Y$ is reduced to a point is already relevant: if $f$ is an automorphism of $\Mod^1(V)$, then the group $\Mod^1(V)$, viewed as a groupoid over a point, comes equipped with two (action-type) cocycles towards $\Mod^1(V)$, given by the identity and $f$. The conclusion in this case is that every automorphism of $\Mod^1(V)$ is a conjugation inside $\Mod^{\pm}(V)$. 

More generally, our work recovers the commensurator rigidity theorem from \cite[Corollary~1.3]{Hen}. Indeed, let $\Gamma_1$ and $\Gamma_2$ be two finite-index subgroups of $\Mod^{\pm}(V)$, and let $f:\Gamma_1\to\Gamma_2$ be an isomorphism. Let $\Gamma'_1\subseteq\Gamma_1$ and $\Gamma'_2\subseteq\Gamma_2$ be finite-index subgroups that are both contained in $\Mod^1(V)$, and such that $f(\Gamma'_1)=\Gamma'_2$. Then $\Gamma'_1$, viewed as a groupoid over a point, comes equipped with two action-type cocycles towards $\Mod^1(V)$, one ranging in $\Gamma'_1$ (given by the identity), and one ranging in $\Gamma'_2$ (given by $f$). Theorem~\ref{theo:main-2} implies that $f_{|\Gamma'_1}$ coincides with the conjugation by an element of $\Mod^{\pm}(V)$. Consequently, the natural map from $\Mod^{\pm}(V)$ to its abstract commensurator is surjective. It is in fact an isomorphism, using that $\Mod^{\pm}(V)$ is ICC for its injectivity.  
\end{rk}

\begin{rk}\label{rk:rotationless-cocycle}
The reason why we are working with cocycles towards the finite-index subgroup $\Mod^1(V)$ from Definition~\ref{de:mod1} is the following. At various places in the proof, we will need to consider subgroupoids that stabilize (in an appropriate sense) a subsurface $\Sigma\subseteq\partial V$, and consider the cocycle to $\Mod(\Sigma)$ obtained by restriction. Lemma~\ref{lem:Mod1} ensures that this restriction cocycle takes its values in $\Mod^0(\Sigma)$, and therefore its image consists of rotationless mapping classes, which is often useful. Arguing in a slightly different way, we could also have avoided the use of $\Mod^1(V)$, and instead impose that the cocycles are \emph{rotationless}, i.e.\ only take rotationless mapping classes as values (which happens for example for cocycles with values in $\Mod^0(\partial V)$). Since the restriction of a rotationless mapping class to a subsurface it preserves is again rotationless, this would have been enough for this purpose. 
\end{rk}

The rest of the section is devoted to the proof of Theorem~\ref{theo:main-2}. Starting from a measured groupoid $\calg$ with two action-type cocycles $\rho_1,\rho_2$ towards $\Mod^1(V)$, we ultimately aim to show that subgroupoids of $\calg$ corresponding to meridian stabilizers for $\rho_1$ - in the precise sense that they are \emph{of meridian type} as in Definition~\ref{de:type} below - are also of meridian type with respect to $\rho_2$. Additionally, we will prove that the property that two subgroupoids stabilize disjoint meridians is also independent of the action-type cocycle we choose. This will be used to build a canonical map $\theta$ from the base space $Y$ of the groupoid $\calg$ to the group of all automorphisms of the disk graph. We will finally appeal to the theorem of Korkmaz and Schleimer \cite{KS} saying that the automorphism group of the disk graph is precisely $\Mod^{\pm}(V)$ to conclude. We make the following definition.

\begin{de}[Subgroupoids of meridian type]\label{de:type}
Let $\calg$ be a measured groupoid over a standard probability space $Y$, and let $\rho:\calg\to\Mod^1(V)$ be a strict cocycle. A measured subgroupoid $\calh$ of $\calg$ is \emph{of meridian type} with respect to $\rho$ if there exists a conull Borel subset $Y^*\subseteq Y$ and a partition $Y^*=\dunion_{i\in I}Y_i$ into at most countably many Borel subsets such that for every $i\in I$, the groupoid $\calh_{|Y_i}$ is equal to the $(\calg_{|Y_i},\rho)$-stabilizer of the isotopy class of a meridian $c_i$. 
\end{de}

When $\calh$ can be written as in Definition~\ref{de:type}, we say that the map $\varphi$ sending every $y\in Y_i$ to the isotopy class of the meridian $c_i$ is a \emph{meridian map} for $(\calh,\rho)$. The essential uniqueness of this map (i.e.\ the fact that, up to measure $0$, it does not depend on the choice of a partition and meridians $c_i$ as above) will follow from Lemmas~\ref{lemma:uniqueness} and~\ref{lemma:uniqueness-disk}.

Likewise, we define the notions of subgroupoids \emph{of nonseparating-meridian type}, and \emph{of separating-meridian type}, by respectively requiring $c_i$ to be nonseparating, or separating. Before completing our characterisation of subgroupoids of meridian type in Proposition~\ref{prop:meridian-type}, we will go  through successive characterisations of subgroupoids of nonseparating-meridian type (Section~\ref{sec:characterization-nonseparating}) and of separating-meridian type (Section~\ref{sec:characterization-separating}).

\subsection{Groupoids with cocycles to a free group, after Adams, Kida}

Throughout the paper, we will work with the following definition.

\begin{de}[Strongly Schottky pairs of subgroupoids]\label{de:strong-schottky}
Let $\calg$ be a measured groupoid over a standard probability space $Y$. A \emph{strongly Schottky pair of subgroupoids} of $\calg$ is a pair $(\cala^1,\cala^2)$ of amenable subgroupoids of $\calg$ of infinite type such that for every Borel subset $U\subseteq Y$ of positive measure, there exists a Borel subset $U'\subseteq U$ of positive measure such that every normal amenable subgroupoid of $\langle \cala^1_{|U'},\cala^2_{|U'}\rangle$ is stably trivial.
\end{de}

We observe that this notion is stable under restrictions: if $(\cala^1,\cala^2)$ is a strongly Schottky pair of subgroupoids of $\calg$, then for every Borel subset $U\subseteq Y$ of positive measure, the pair $(\cala^1_{|U},\cala^2_{|U})$ is a strongly Schottky pair of subgroupoids of $\calg_{|U}$. In addition, this notion is stable under stabilization: given a pair $(\cala^1,\cala^2)$ of subgroupoids of $\calg$, and a partition $Y=\dunion_{i\in I}Y_i$ into at most countably many Borel subsets, if $(\cala^1_{|Y_i},\cala^2_{|Y_i})$ is a strongly Schottky pair of subgroupoids of $\calg_{|Y_i}$ for every $i\in I$, then $(\cala^1,\cala^2)$ is a strongly Schottky pair of subgroupoids of $\calg$.

Notice that the last conclusion implies in particular that $\langle\cala^1_{|U'},\cala^2_{|U'}\rangle$ is nonamenable. So the existence of a strongly Schottky pair of subgroupoids of $\calg$ forces $\calg$ to be everywhere nonamenable.

Definition~\ref{de:strong-schottky} is a strengthening of the notion of a \emph{Schottky pair of subgroupoids} from \cite[Definition~13.1]{GH}, which only required the groupoid $\langle \cala^1_{|U},\cala^2_{|U}\rangle$ to be nonamenable. The following lemma is a variation over arguments of Adams \cite[Section~3]{Ada} and Kida \cite[Lemma~3.20]{Kid}, and gives the main example of a strongly Schottky pair of subgroupoids. 

\begin{lemma}\label{lemma:free}
Let $G$ be a countable group, and let $g,h\in G$ be two elements that generate a nonabelian free subgroup $F$ of $G$. Let $\calg$ be a measured groupoid over a standard probability space $Y$, equipped with a strict action-type cocycle $\rho:\calg\to G$.

Then $(\rho^{-1}(\langle g\rangle),\rho^{-1}(\langle h\rangle))$ is a strongly Schottky pair of subgroupoids of $\calg$ (in particular $\calg$ is everywhere nonamenable).

Moreover, for every positive measure subset $U\subseteq Y$, every normal amenable subgroupoid of $\rho^{-1}(F)_{|U}$ is stably trivial.
\end{lemma}

In the following proof, whenever $\Delta$ is a Polish space, the set $\Prob(\Delta)$ of all Borel probability measures on $\Delta$ is equipped with the topology generated by the maps $\mu\mapsto\int_X fd\mu$, where $f$ varies over the set of all real-valued bounded continuous functions. When $\Delta$ is compact, this is nothing but the weak-$\ast$ topology coming from the duality given by the Riesz--Markov--Kakutani theorem. When $\Delta$ is a countable discrete space, this is nothing but the topology of pointwise convergence. The reader is refered to \cite[Section~17.E]{Kec} for more information and basic facts regarding the Borel structure on $\Prob(\Delta)$ which justify the measurability of all maps in the following proof. 

\begin{proof}[{Proof of Lemma~\ref{lemma:free}}]
As $\langle g\rangle$ and $\langle h\rangle$ are amenable subgroups of $G$ and $\rho$ has trivial kernel, the subgroupoids $\rho^{-1}(\langle g\rangle)$ and $\rho^{-1}(\langle h\rangle)$ are amenable, see \cite[Corollary~3.39]{GH}. As $\langle g\rangle$ and $\langle h\rangle$ are infinite and $\rho$ is action-type, the subgroupoids $\rho^{-1}(\langle g\rangle)$ and $\rho^{-1}(\langle h\rangle)$ are of infinite type.

Now it is enough to prove that if $U\subseteq Y$ is a Borel subset of positive measure, and $\cala$ is a normal amenable subgroupoid of either $\langle \rho^{-1}(\langle g\rangle)_{|U},\rho^{-1}(\langle h\rangle)_{|U}\rangle$ or of $\rho^{-1}(F)_{|U}$, then $\cala$ is stably trivial.

Let $T$ be the Cayley tree of the free group $F=\langle g,h\rangle$, with respect to the generating set $\{g,h\}$. The $F$-action on $T$ by isometries extends to an $F$-action on $\partial_\infty T$ by homeomorphisms. As $\cala$ is amenable, and as $\rho(\cala)$ is contained in the group $F$ which acts by homeomorphisms on the compact metrizable space $\partial_\infty T$, we can apply \cite[Proposition~4.14]{Kid-survey} and get an $(\cala,\rho)$-equivariant Borel map $U\to \Prob(\partial_\infty T)$.

Let $\mathfrak{F}$ be the set of all Borel subsets $W\subseteq U$ such that there exists a Borel map $\mu:W\to\Prob(\partial_\infty T)$ which is stably $(\cala_{|W},\rho)$-equivariant and such that for every $y\in W$, the support of the measure $\mu(y)$ has cardinality at least $3$. The set $\mathfrak{F}$ is stable under countable unions. Therefore, by Lemma~\ref{lemma:zorn}, it admits a maximal element $U_1$ (in the sense that every $W\in\mathfrak{F}$ has a conull Borel subset contained in $U_1$).


Let us now fix a Borel subset $U_1\subseteq U$ as above. We first prove that $\cala_{|U_1}$ is stably trivial. Up to partitioning $U_1$ into at most countably many Borel subsets, we can assume that the map $\mu_1$ is $(\cala_{|U_1},\rho)$-equivariant (and not just stably equivariant). For every $y\in U_1$, the probability measure $\mu(y)\otimes\mu(y)\otimes\mu(y)$ on $(\partial_\infty T)^3$ gives positive measure to the $F$-invariant subset $(\partial_\infty T)^{(3)}$ made of pairwise distinct triples. Thus, after restricting this measure to $(\partial_\infty T)^{(3)}$ and renormalizing to turn this restricted measure into a probability measure, we get an $(\cala_{|U_1},\rho)$-equivariant Borel map $U_1\to\Prob((\partial_\infty T)^{(3)})$. 

Now, denoting by $V(T)$ the vertex set of $T$, there is a natural $F$-equivariant \emph{barycenter} map $(\partial_\infty T)^{(3)}\to V(T)$. Indeed, given $(\xi_1,\xi_2,\xi_3)\in (\partial_\infty T)^{(3)}$, the three geodesic lines $\ell_1,\ell_2,\ell_3$ (where $\ell_i$ joins $\xi_{i-1}$ to $\xi_{i+1}$, with indices considered modulo $3$) meet in a single vertex of $T$. This vertex is the barycenter of $(\xi_1,\xi_2,\xi_3)$. It depends continuously on $(\xi_1,\xi_2,\xi_3)$, being in fact locally constant.

By pushing the probability measures on $(\partial_\infty T)^{(3)}$ through this barycenter map, we get an $(\cala_{|U_1},\rho)$-equivariant Borel map $U_1\to\Prob(V(T))$. Let $\calp_{<\infty}(V(T))$ be the set of all nonempty finite subsets of $V(T)$. As $V(T)$ is countable, there is also a natural $F$-equivariant Borel map $\Prob(V(T))\to\calp_{<\infty}(V(T))$, sending a probability measure $\nu$ to the finite subset of $V(T)$ made of all vertices that have maximal $\nu$-measure. We thus derive an $(\cala_{|U_1},\rho)$-equivariant Borel map $\phi:U_1\to\calp_{<\infty}(V(T))$. As $\calp_{<\infty}(V(T))$ is countable, we can then find a Borel partition $U_1=\dunion_{i\in I}U_{1,i}$ into at most countably many Borel subsets such that for every $i\in I$, the map $\phi_{|U_{1,i}}$ is constant, with value a nonempty finite set $\calf_i$ of vertices of $T$. In other words, there exists a conull Borel subset $U_{1,i}^*\subseteq U_{1,i}$ such that $\rho(\cala_{|U_{1,i}^*})$ is contained in the $F$-stabilizer of $\calf_i$. As this stabilizer is trivial and $\rho$ has trivial kernel, it follows that $\cala_{|U_1}$ is stably trivial.

We will now prove that $U_2=U\setminus U_1$ is a null set, which will conclude the proof of the lemma. So assume towards a contradiction that $U_2$ has positive measure. We know that there exists an $(\cala_{|U_2},\rho)$-equivariant Borel map $\mu:U_2\to\Prob(\partial_\infty T)$, and that for every such map and almost every $y\in U_2$, the support of $\mu(y)$ has cardinality at most $2$. Let $\calp_{\le 2}(\partial_\infty T)$ be the set of all nonempty subsets of $\partial_\infty T$ of cardinality at most $2$. As in \cite[Lemma~3.2]{Ada}, we can thus find an $(\cala_{|U_2},\rho)$-equivariant Borel map $\theta_{\max}:U_2\to\calp_{\le 2}(\partial_\infty T)$ which is maximal in the sense that for every other $(\cala_{|U_2},\rho)$-equivariant Borel map $\theta:U_2\to\calp_{\le 2}(\partial_\infty T)$ and a.e.\ $y\in Y$, one has $\theta(y)\subseteq\theta_{\max}(y)$. Being canonical, the map $\theta_{\max}$ is then equivariant under the groupoid $\langle \rho^{-1}(\langle g\rangle)_{|U_2},\rho^{-1}(\langle h\rangle)_{|U_2}\rangle$ which normalizes $\cala_{|U_2}$ (compare also with the proof of Lemma~\ref{lemma:crs-normal} below where a similar argument is detailed). Recall that the groupoid $\rho^{-1}(\langle g\rangle)_{|U_2}$ is amenable and of infinite type. Therefore, repeating the argument from the present proof shows that there exists a maximal $(\rho^{-1}(\langle g\rangle)_{|U_2},\rho)$-equivariant Borel map $U_2\to\calp_{\le 2}(\partial_\infty T)$, and this must then be the constant map with value $\{g^{-\infty},g^{+\infty}\}$. Likewise, the constant map with value $\{h^{-\infty},h^{+\infty}\}$ is the maximal $(\rho^{-1}(\langle h\rangle)_{|U_2},\rho)$-equivariant Borel map $U_2\to\calp_{\le 2}(\partial_\infty T)$. As $\{g^{-\infty},g^{+\infty}\}\cap\{h^{-\infty},h^{+\infty}\}=\emptyset$, we have reached a contradiction. This completes our proof. 
\end{proof}

\subsection{Canonical reduction sets, after Kida}
In this section, we review work of Kida \cite[Chapter~4]{Kid-memoir} regarding groupoids with cocycles towards a surface mapping class group. Since our terminology slightly differs from Kida's, we recall proofs for the convenience of the reader. We will introduce a notion of canonical reduction multicurve for a groupoid equipped with a cocycle towards a surface mapping class group, generalizing the classical notion for subgroups of $\Mod(\Sigma)$ that we recalled in Section~\ref{sec:bg}. 

We also mention that the results in this section can also be viewed as a special case of those in \cite[Section~3.6]{HH2}, applied by taking for $\mathbb{P}$ the set of all elementwise stabilizers of collections of curves on the surface, but we believe it is useful to have the arguments specified in our context. In the whole section, we let $\Sigma$ be a (possibly disconnected) orientable surface of finite type, i.e.\ $\Sigma$ is obtained from the disjoint union of finitely many closed connected orientable surfaces by removing at most finitely many points. Recall that $\Mod^0(\Sigma)$ is the group of all isotopy classes of orientation-preserving diffeomorphisms of $\Sigma$ that do not permute the connected components of $\Sigma$, and act trivially on the homology mod $3$ of each connected component; in other words $\Mod^0(\Sigma)=\Mod^0(\Sigma_1)\times\dots\times\Mod^0(\Sigma_k)$, where $\Sigma_1,\dots,\Sigma_k$ are the connected components of $\Sigma$. 

\begin{de}[Irreducibility]
Let $\calg$ be a measured groupoid over a standard probability space $Y$, equipped with a strict cocycle $\rho:\calg\to\Mod^0(\Sigma)$. 

We say that $(\calg,\rho)$ is \emph{reducible} if there exist a Borel subset $U\subseteq Y$ of positive measure and an essential simple closed curve $c$ on $\Sigma$ such that the isotopy class of $c$ is $(\calg_{|U},\rho)$-invariant. 

Otherwise, we say that $(\calg,\rho)$ is \emph{irreducible}.
\end{de}

\begin{de}[Canonical reduction set]
Let $\calg$ be a measured groupoid over a standard probability space $Y$, equipped with a strict cocycle $\rho:\calg\to\Mod^0(\Sigma)$. A (possibly infinite) set $\calc$ of isotopy classes of essential simple closed curves on $\Sigma$ is a \emph{canonical reduction set} for $(\calg,\rho)$ if
\begin{enumerate}
\item every $c\in\calc$ is $(\calg,\rho)$-invariant, and
\item for every Borel subset $U\subseteq Y$ of positive measure, every isotopy class $c'$ of essential simple closed curves which is $(\calg_{|U},\rho)$-invariant belongs to $\calc$.
\end{enumerate}
\end{de}

Note that $(\calg,\rho)$ is irreducible if and only if $\emptyset$ is a canonical reduction set for $\calg$. Notice also that if a canonical reduction set for $(\calg,\rho)$ exists, then it is unique (because it is the set of all $(\calg,\rho)$-invariant isotopy classes of essential simple closed curves). We also observe that if $\calc$ is a canonical reduction set for $(\calg,\rho)$, then for every positive measure Borel subset $U\subseteq X$, the set $\calc$ is also a canonical reduction set for $(\calg_{|U},\rho)$. The latter observation will often allow us to restrict to a positive measure Borel subset of the base space, without having to worry about changing the canonical reduction set of the groupoid (and cocycle) under consideration.

Lemma~\ref{lemma:crs} below shows that up to a countable Borel
partition of the base space, canonical reduction sets always
exist. For this, we need two lemmas.

The first is immediate from Lemma~\ref{lem:fixing-filling}, noting that if
an element of $\Mod^1(\Sigma)$ preserves the subsurface $S_\calc$,
then the restriction lies in $\Mod^0(S_\calc)$ by
Lemma~\ref{lem:Mod1}.
\begin{lemma}\label{lem:stab-filled-subsurface}
  Let $\Sigma$ be a (possibly disconnected) surface of finite type, let $\calc$ be a set of isotopy classes of essential simple closed curves, and let $S_\calc$ be the subsurface filled by $\calc$ (compare Lemma~\ref{lem:subsurface-filled}).

Then the elementwise stabilizer of $\calc$ in $\Mod^1(\Sigma)$ is the subgroup of $\Mod^1(\Sigma)$ consisting of all mapping classes that have a representative supported on the complement of $S$.
\end{lemma}

The second is an immediate consequence of Lemma~\ref{lem:stab-filled-subsurface} and Lemma~\ref{lem:subsurface-chains}:
\begin{cor}\label{cor:stabchainbound}
Let $\Sigma$ be a (possibly disconnected) surface of finite type. There is a bound on the size $k$ of a chain $\calc_1\subseteq\cdots\subseteq\calc_k$ of sets of isotopy classes of essential simple closed curves on $\Sigma$ such that, for every $i\in\{1,\dots,k-1\}$, the elementwise stabilizer of $\calc_{i+1}$ in $\Mod^0(\Sigma)$ is a proper subgroup of the elementwise stabilizer of $\calc_i$. 
\end{cor}

\begin{lemma}\label{lemma:crs}
Let $\calg$ be a measured groupoid over a standard probability space $Y$, equipped with a strict cocycle $\rho:\calg\to\Mod^0(\Sigma)$. 

Then there exist a partition $Y=\dunion_{i\in I} Y_i$ into at most countably many Borel subsets such that for every $i\in I$, $(\calg_{|Y_i},\rho)$ has a canonical reduction set.
\end{lemma}

\begin{proof}
Let $\mathfrak{F}$ be the set of all Borel subsets $U\subseteq X$ which admit a partition $U=\dunion_{i\in I}U_i$ into at most countably many Borel subsets, such that for every $i\in I$, there exists an essential simple closed curve $c_i$ on $\Sigma$ whose isotopy class is $(\calg_{|U_i},\rho)$-invariant. The set $\mathfrak{F}$ is stable under countable unions. Therefore, by Lemma~\ref{lemma:zorn}, it has a maximal element $Y'_0$, i.e.\ such that every $W\in\mathfrak{F}$ has a conull Borel subset contained in $Y'_0$.


Let $Y_0=Y\setminus Y'_0$. The maximality of $Y'_0$ ensures that $(\calg_{|Y_0},\rho)$ is irreducible. 

The definition of $Y'_0$ allows us to fix a partition $Y'_0=\dunion_{i\in I_0}Y_i$ into at most countably many Borel subsets such that for every $i\in I_0$, there exists an essential simple closed curve $c_i$ whose isotopy class is $(\calg_{|Y_i},\rho)$-invariant. For every $i\in I_0$, let $\calc_i$ be the (nonempty) set of all isotopy classes of essential simple closed curves on $\Sigma$ that are $(\calg_{|Y_i},\rho)$-invariant. Let $\Gamma_{\calc_i}$ be the elementwise stabilizer of $\calc_i$ in $\Mod^0(\Sigma)$: this is a proper subgroup of $\Mod^0(\Sigma)$ because $\calc_i\neq\emptyset$. Repeating the above argument, for every $i\in I_0$, there exists a Borel partition $Y_i=Y_{i,0}\dunion Y_{i,0}'$ such that 
\begin{enumerate}
\item for every Borel subset $U\subseteq Y_{i,0}$ of positive measure, every $(\calg_{|U},\rho)$-invariant isotopy class of essential simple closed curve belongs to $\calc_i$,
\item there exists a partition $Y'_{i,0}=\dunion_{j\in J_{i}}Y_{i,j}$ into at most countably many Borel subsets such that for every $j\in J_i$,  there exists an essential simple closed curve on $\Sigma$ whose isotopy class is $(\calg_{|Y_{i,j}},\rho)$-invariant, but does not belong to $\calc_i$. 
\end{enumerate}   
For every $j\in J_i$, we then let $\calc_{i,j}$ be the set of all isotopy classes of essential simple closed curves on $\Sigma$ that are $(\calg_{|Y_{i,j}},\rho)$-invariant. And we let $\Gamma_{\calc_{i,j}}$ be the elementwise stabilizer of $\calc_{i,j}$ in $\Mod^0(\Sigma)$. We observe that $\Gamma_{\calc_{i,j}}$ is a proper subgroup of $\Gamma_{\calc_i}$. Indeed, there exists a conull Borel subset $Y_i^*\subseteq Y_i$ such that $\rho(\calg_{|Y_i^*})\subseteq\Gamma_{\calc_i}$. If $\Gamma_{\calc_{i,j}}=\Gamma_{\calc_i}$, then every curve in $\calc_{i,j}\setminus\calc_i$ is $\Gamma_{\calc_i}$-invariant, and therefore $(\calg_{|Y_i},\rho)$-invariant, contradicting the definition of $\calc_i$. This contradiction shows that $\Gamma_{\calc_{i,j}}\subsetneq\Gamma_{\calc_i}$.

We now repeat the above procedure inductively. By Corollary~\ref{cor:stabchainbound}, there is a bound, only depending on the topology of $\Sigma$, on a chain (for inclusion) of collections of curves on $\Sigma$ with pairwise distinct elementwise stabilizers in $\Mod^0(\Sigma)$. Thus we attain a partition of $Y$ with the required properties after finitely many iterations of the above procedure. This completes the proof.    
\end{proof}

The following lemma justifies the \emph{canonicity} of a canonical reduction set.

\begin{lemma}\label{lemma:crs-normal}
Let $\calg$ be a measured groupoid over a standard probability space $Y$, equipped with a strict cocycle $\rho:\calg\to\Mod^0(\Sigma)$, and let $\calh$ be a measured subgroupoid of $\calg$. Assume that $(\calh,\rho)$ has a canonical reduction set $\calc$.

Then for every measured subgroupoid $\calh'$ of $\calg$ that normalizes $\calh$, the set $\calc$ is $(\calh',\rho)$-invariant. In other words, denoting by $\Stab(\calc)$ the global stabilizer of $\calc$ in $\Mod^0(\Sigma)$, there exists a conull Borel subset $Y^*\subseteq Y$ such that $\rho(\calh'_{|Y^*})\subseteq\Stab(\calc)$.
\end{lemma}

\begin{proof}
Since $\calh'$ normalizes $\calh$, there exists a covering of $\calh'$ by countably many bisections $B_n$ that all leave $\calh$ invariant. Up to subdividing the bisections $B_n$, we will assume that for every $n\in\mathbb{N}$, the $\rho$-image of $B_n$ is a single element $\gamma_n\in\Mod^0(\Sigma)$. For every $n\in\mathbb{N}$, we let $U_n$ and $V_n$ be the source and range of $B_n$. 

Let $c\in\calc$ and $n\in\mathbb{N}$. Then $c$ is $(\calh_{|U_n},\rho)$-invariant, so by $B_n$-invariance of $\calh$, the isotopy class $\gamma_n c$ is $(\calh_{|V_n},\rho)$-invariant.
If $V_n$ has positive measure, the maximality condition in the definition of a canonical reduction set ensures that $\gamma_n c\in\calc$. By reversing the arrows in the bisection $B_n$, we also derive that $\gamma_nc\notin\calc$ if $c\notin\calc$. 

Let now $Y^*\subseteq Y$ be a conull Borel subset which avoids each of the countably many subsets $U_n$ and $V_n$ of zero measure. The above ensures that $\rho(\calh'_{|Y^*})\subseteq\Stab(\calc)$. This concludes our proof.
\end{proof}

Recall from Lemma~\ref{lem:subsurface-filled} that given a (possibly
infinite) set $\calc$ of isotopy classes of essential simple closed
curves on $\Sigma$, there is a unique subsurface $S$ filled by
$\calc$. The multicurve $X$, obtained from $\partial S$ by only
keeping one curve in each isotopy class, is called the \emph{boundary
  multicurve} of $\calc$.

\begin{cor}\label{cor:reducibility-normal}
Let $\calg$ be a measured groupoid over a standard probability space $Y$, equipped with a strict cocycle $\rho:\calg\to\Mod^0(\Sigma)$. Let $\calh,\calh'\subseteq\calg$ be measured subgroupoids. Assume that $\calh$ is stably normalized by $\calh'$, and that for every Borel subset $U\subseteq Y$ of positive measure, one has $\rho(\calh_{|U})\neq\{1\}$.

If $(\calh,\rho)$ is reducible, then so is $(\calh',\rho)$.
\end{cor}

\begin{proof}
  Since $(\calh,\rho)$ is reducible, we can find a Borel subset $U\subseteq Y$ of positive measure such that $(\calh_{|U},\rho)$ has a nonempty canonical reduction set $\calc$. As $\rho(\calh_{|U})\neq\{1\}$, the set $\calc$ does not fill $\Sigma$ (Lemma~\ref{lem:fixing-filling}), so the boundary multicurve $X$ of $\calc$ is nonempty. Up to restricting to a Borel subset of $U$ of positive measure  (which does not change the canonical reduction set of $(\calh_{|U},\rho)$), we can assume that $\calh_{|U}$ is normalized by $\calh'_{|U}$. Lemma~\ref{lemma:crs-normal} ensures that $\calc$ is $(\calh'_{|U},\rho)$-invariant. In particular $X$ is $(\calh'_{|U},\rho)$-invariant, showing that $(\calh',\rho)$ is reducible.
\end{proof}

When $\calc$ is the canonical reduction set for $(\calh,\rho)$, the boundary multicurve $X$ of $\calc$ will be called the \emph{canonical reduction multicurve} of $(\calh,\rho)$. A connected component $S$ of $\Sigma\setminus X$ is then called \emph{active} for $(\calh,\rho)$ if it contains an essential simple closed curve whose isotopy class does not belong to $\calc$, and \emph{inactive} for $(\calh,\rho)$ otherwise (because in the latter case, every element in the elementwise stabilizer of $\calc$ acts trivially on $S$).

We give a few examples of active and inactive subsurfaces in the case that the essential image of $\rho$ is a cyclic subgroup generated by $\varphi$ and $\rho$ has trivial kernel.
\begin{enumerate}[i)]
\item If $\varphi$ is a partial pseudo-Anosov supported on a connected subsurface $Z \subset \Sigma$, possibly composed with Dehn twists about curves contained in $\partial Z$, then the canonical reduction multicurve is $\partial Z$, and $Z$ is the only active complementary component. 
\item If $\varphi$ is a Dehn twist about a curve $\alpha$, then the canonical reduction multicurve is $\alpha$, and all complementary components are inactive.
\end{enumerate}

We also observe that if $\rho:\calg\to\Mod^0(\Sigma)$ is an action-type cocycle, and if $H\subseteq\Mod^0(\Sigma)$ is a subgroup, then $\rho^{-1}(H)$ has a canonical reduction set, equal to the canonical reduction set of $H$ (in the sense recalled in Section~\ref{sec:bg}) -- in particular $\rho^{-1}(H)$ has a canonical reduction multicurve, equal to that of $H$. Indeed, every curve on $\Sigma$ whose isotopy class is fixed by $H$, is also fixed by $\rho^{-1}(H)$ (up to isotopy). And conversely, let $c$ be a curve whose isotopy class is fixed by $\rho^{-1}(H)_{|U}$ for some positive measure Borel subset $U\subseteq Y$, and let $h\in H$. Since $\rho$ is action-type, there exists $n\neq 0$ such that $h^n$ is in the essential image of the restriction of $\rho$ to $\calg_{|U}$. In particular $h^n$ fixes $c$, so $h$ fixes $c$ as we are working in the rotationless subgroup $\Mod^0(\Sigma)$. So $c$ is fixed by $H$, i.e.\ $c$ belongs to the canonical reduction set of $H$.

\subsection{Exploiting amenable normalized subgroupoids, after Kida}

The following statement, which was established by Kida in \cite[Section~4.4.1]{Kid}, will be used extensively in the remainder of this section, applied either to $\partial V$ or to subsurfaces of $\partial V$. We include a proof to explain how to deal with disconnected subsurfaces.

\begin{lemma}[Kida]\label{lemma:kida-ia}
Let $\Sigma$ be a (possibly disconnected) surface of finite type, so that every connected component has negative Euler characteristic. Let $\calg$ be a measured groupoid, equipped with a strict cocycle $\rho:\calg\to\Mod^0(\Sigma)$. Let $\calh$ be a measured subgroupoid of $\calg$ such that $\rho_{|\calh}$ has trivial kernel. 

If $\calh$ stably normalizes an amenable subgroupoid $\cala$ of $\calg$, with $(\cala,\rho)$ irreducible, then $\calh$ is amenable. 
\end{lemma}

In the following proof, we will make use of the (compact) space $\mathrm{PML}$ of projective measured laminations, and the (measurable) subspace $\mathrm{AL}$ of \emph{arational} (i.e.\ minimal and filling) laminations. These notions arose in Thurston's work on surfaces; standard references include \cite{CB,FLP}. Let us also mention the dictionary with Kida's work for comparison. Via the dictionary between measured laminations and measured foliations (see e.g.\ \cite{Lev}), our space $\mathrm{PML}$ is isomorphic to the space $\mathcal{PMF}$ of projective measured foliations. And $\mathrm{AL}$ is the same, in Kida's notation, as the subspace $\mathcal{MIN}$ consisting of minimal foliations.

\begin{proof}
Up to a countable Borel partition of the base space $Y$ of $\calg$ (which does not affect the conclusion), we will assume that $\calh$ normalizes $\cala$.

Let $\Sigma_1,\dots,\Sigma_k$ be the connected components of $\Sigma$. Then $\Mod^0(\Sigma)$ decomposes as $\Mod^0(\Sigma)=\Mod^0(\Sigma_1)\times\dots\times\Mod^0(\Sigma_k)$. For $i\in\{1,\dots,k\}$, let $\rho_i:\calg\to\Mod^0(\Sigma_i)$ be the cocycle obtained by post-composing $\rho$ with the $i^{\text{th}}$ projection.

Let $i\in\{1,\dots,k\}$. Then $\Mod^0(\Sigma_i)$ acts on the compact metrizable space $\PML(\Sigma_i)$ of projective measured laminations on $\Sigma_i$. As $\cala$ is amenable, there exists an $(\cala,\rho_i)$-equivariant Borel map $\mu:Y\to\Prob(\PML(\Sigma_i))$. The space $\PML(\Sigma_i)$ has a $\Mod(\Sigma_i)$-invariant Borel partition into the subspace $\AL_i$ made of arational laminations, and the subspace $\NAL_i$ made of non-arational laminations.

Let us first assume towards a contradiction that there exists a Borel subset $U\subseteq Y$ of positive measure such that for all $y\in U$, the measure $\mu(y)$ gives positive measure to $\NAL_i$. After restricting $\mu(y)$ to $\NAL_i$ and renormalizing it to get a probability measure, we obtain an $(\cala_{|U},\rho_i)$-equivariant Borel map $U\to\Prob(\NAL_i)$. Let $\calp_{<\infty}(\calc(\Sigma_i))$ be the countable set of all nonempty finite sets of isotopy classes of essential simple closed curves on $\Sigma_i$. There is a $\Mod(\Sigma_i)$-equivariant map $\NAL_i\to\calp_{<\infty}(\calc(\Sigma_i))$, sending a lamination to the union of all simple closed curves it contains together with all boundaries of the subsurfaces it fills. We thus get an $(\cala_{|U},\rho_i)$-equivariant Borel map $U\to\Prob(\calp_{<\infty}(\calc(\Sigma_i)))$. As $\calp_{<\infty}(\calc(\Sigma_i))$ is countable, there is also a $\Mod(\Sigma_i)$-equivariant map $\Prob(\calp_{<\infty}(\calc(\Sigma_i)))\to\calp_{<\infty}(\calc(\Sigma_i))$, sending a probability measure $\nu$ to the union of all finite sets with maximal $\nu$-measure. In summary, we have found an $(\cala_{|U},\rho_i)$-equivariant Borel map $U\to\calp_{<\infty}(\calc(\Sigma_i))$. Let $V\subseteq U$ be a Borel subset of positive measure where this map is constant, with value a finite set $\calf$. As we are working in the finite-index subgroup $\Mod^0(\Sigma_i)$, every curve in $\calf$ is $(\cala_{|V},\rho_i)$-invariant, contradicting the irreducibility of $(\cala,\rho)$.  

Therefore $\mu$ determines an $(\cala,\rho_i)$-equivariant Borel map $Y\to\Prob(\AL_i)$. Klarreich's description \cite{Kla} of the boundary $\partial_\infty\calc_i$ of the curve graph of $\Sigma_i$ yields a continuous $\Mod(\Sigma_i)$-equivariant map $\AL_i\to\partial_\infty\calc_i$, so we get an $(\cala,\rho_i)$-equivariant Borel map $Y\to\Prob(\partial_\infty\calc_i)$. Denoting by $(\partial_\infty\calc_i)^{(3)}$ the space of pairwise distinct triples, Kida proved in \cite[Section~4.1]{Kid-memoir} the existence of a $\Mod(\Sigma_i)$-equivariant Borel map $(\partial_\infty\calc_i)^{(3)}\to \calp_{<\infty}(\calc(\Sigma_i))$. Using again the irreducibility of $(\cala,\rho)$, together with an Adams-type argument as in the proof of Lemma~\ref{lemma:free}, we deduce that there exists a Borel map $Y\to\calp_{\le 2}(\partial_\infty\calc_i)$ which is both $(\cala,\rho_i)$-equivariant and $(\calh,\rho_i)$-equivariant.

Combining all these maps as $i$ varies in $\{1,\dots,k\}$ yields an $(\calh,\rho)$-equivariant Borel map $$Y\to\calp_{\le 2}(\partial_\infty\calc_1)\times\dots\times\calp_{\le 2}(\partial_\infty\calc_k).$$ For every $i\in\{1,\dots,k\}$, the action of $\Mod(\Sigma_i)$ on $\partial_\infty\calc_i$ is Borel amenable \cite{Kid-memoir,Ham}, and therefore so is the action of $\Mod^0(\Sigma)$ on $\calp_{\le 2}(\partial_\infty\calc_1)\times\dots\times\calp_{\le 2}(\partial_\infty\calc_k)$ (see e.g.\ \cite[Section~3.4.1]{HH2} for the relevant background). As $\rho_{|\calh}$ has trivial kernel, it then follows from \cite[Proposition~3.38]{GH} (originally due to Kida \cite[Proposition~4.33]{Kid-memoir}) that $\calh$ is amenable.
\end{proof}

\subsection{Uniqueness statements}

\begin{lemma}\label{lemma:uniqueness}
Let $\calg$ be a measured groupoid over a standard probability space $Y$, equipped with a strict action-type cocycle $\rho:\calg\to\Mod^1(V)$. Let $\calh$ be a measured subgroupoid of $\calg$. 

Let $c$ be a nonseparating meridian, and let $c'$ be an essential simple closed curve on $\partial V$. Assume that there exists a Borel subset $U\subseteq Y$ of positive measure such that $\calh_{|U}$ is equal to the $(\calg_{|U},\rho)$-stabilizer of the isotopy class of $c$, and the isotopy class of $c'$ is $(\calh_{|U},\rho)$-invariant. 

Then $c'=c$ (up to isotopy).
\end{lemma}

\begin{proof}
The stabilizer of $c$ in $\Mod^1(V)$ contains an element $g$ which restricts to a pseudo-Anosov element on $\partial V\setminus c$ (Lemma~\ref{lemma:pA-in-complement}). The groupoid $\rho^{-1}(\langle g\rangle)_{|U}$ is contained in $\calh_{|U}$, and it is of infinite type since $\rho$ is action-type. Therefore $c'$ is fixed by some positive power of $g$, which implies that $c'=c$ up to isotopy.
\end{proof}

The following is a version of Lemma~\ref{lemma:uniqueness} for
separating meridians.

\begin{lemma}\label{lemma:uniqueness-disk}
Let $\calg$ be a measured groupoid over a standard probability space $Y$, equipped with a strict action-type cocycle $\rho:\calg\to\Mod^1(V)$. Let $\calh$ be a measured subgroupoid of $\calg$. 

Let $c,c'$ be two separating meridians. Assume that there exists a Borel subset $U\subseteq Y$ of positive measure such that $\calh_{|U}$ is equal to the $(\calg_{|U},\rho)$-stabilizer of the isotopy class of $c$, and the isotopy class of $c'$ is $(\calh_{|U},\rho)$-invariant. 

Then $c=c'$ (up to isotopy). 
\end{lemma}

\begin{proof}
By Corollary~\ref{cor:sep-meridian}, the stabilizer of $c$ in $\Mod^1(V)$ contains an element $g$ such that for every $n\neq 0$, the curve $c$ is (up to isotopy) the only essential separating meridian whose isotopy class is fixed by $g^n$. The groupoid $\rho^{-1}(\langle g\rangle)_{|U}$ is contained in $\calh_{|U}$, and it is of infinite-type since $\rho$ is action-type. Therefore $c'$ is fixed by some positive power of $g$, which by our choice of $g$ implies that $c'=c$ (up to isotopy).   
\end{proof}

\subsection{Property~$\pnsep$ and subgroupoids of non-separating meridian type}\label{sec:characterization-unilateral}

We make the following definition (see Definition~\ref{de:strong-schottky} for the notion of a strongly Schottky pair of subgroupoids).

\begin{de}[Product-like subgroupoid]\label{de:product-like}
A measured groupoid $\calp$ is \emph{product-like} if there exist two subgroupoids $\calp_1,\calp_2\subseteq\calp$ which are both stably normal in $\calp$, such that for every $i\in\{1,2\}$, the groupoid $\calp_i$ contains a strongly Schottky pair of subgroupoids $(\cala_i^1,\cala_i^2)$, with $\cala_i^1$ and $\cala_i^2$ both stably normalized by $\calp_{3-i}$.
\end{de}

Notice that this notion is stable under restrictions and stabilization. In the terminology from \cite[Definition~13.5]{GH}, the subgroupoids $\calp_1$ and $\calp_2$ form a pseudo-product. One difference between our definition and \cite[Definition~13.5]{GH} is that we are working with strongly Schottky pairs of subgroupoids, while \cite[Definition~13.5]{GH} is phrased using the weaker notion of Schottky pairs of subgroupoids. Also, we are further imposing that $\calp_1$ and $\calp_2$ are stably normal in an ambient groupoid $\calp$.

We now introduce the following properties, which will be useful in order to detect subgroupoids of nonseparating-meridian type.

\begin{de}\label{de:pnsep}
Let $\calg$ be a measured groupoid, and let $\cala,\calh$ be measured subgroupoids of $\calg$, with $\cala\subseteq\calh$.
\begin{enumerate}
\item We say that the pair $(\calh,\cala)$ satisfies \emph{Property~$\qnsep$} if the following conditions hold:
\begin{enumerate}
\item $\calh$ is everywhere nonamenable;
\item $\cala$ is amenable, of infinite type, and stably normal in $\calh$;
\item if $\calb$ is a stably normal amenable subgroupoid of $\calh$, then $\calb$ is stably contained in $\cala$;
\item if $\calh'$ is another subgroupoid of $\calg$ which is everywhere nonamenable and contains a stably normal amenable subgroupoid of infinite type, and if $\calh$ is stably contained in $\calh'$, then $\calh$ is stably equal to $\calh'$;
\item for every Borel subset $U\subseteq Y$ of positive measure, the groupoid $\calh_{|U}$ is not contained in any product-like subgroupoid of $\calg_{|U}$. 
\end{enumerate}
\item We say that $\calh$ satisfies \emph{Property~$\pnsep$} if there exists a measured subgroupoid $\cala\subseteq\calh$ such that $(\calh,\cala)$ satisfies Property~$\qnsep$.
\end{enumerate}
\end{de}

\begin{rk}\label{rk:unique-a}
These properties are stable under restrictions and stabilization. Also, if $\calh$ satisfies Property~$\pnsep$, then a subgroupoid $\cala\subseteq\calh$ such that $(\calh,\cala)$ satisfies Property~$\qnsep$ is ``stably unique'' in the following sense: if $\cala$ and $\cala'$ are two such subgroupoids, there exist a conull Borel subset $Y^*\subseteq Y$ and a partition $Y^*=\dunion_{i\in I}Y_i$ into at most countably many Borel subsets such that for every $i\in I$, one has $\cala_{|Y_i}=\cala'_{|Y_i}$. Indeed, this is a consequence of Assumptions~(b) and~(c) from the definition.  
\end{rk}

The goal of the present section is to prove that subgroupoids of nonseparating-meridian type with respect to an action-type cocycle $\calg\to\Mod^1(V)$ (in the sense of Definition~\ref{de:type} and the paragraph below it) satisfy Property~$\pnsep$.

\begin{prop}\label{prop:characterization-vertices}
Let $\calg$ be a measured groupoid over a standard probability space $Y$, equipped with a strict action-type cocycle $\rho:\calg\to\Mod^1(V)$. Let $c$ be a nonseparating meridian, let $\calh$ be the $(\calg,\rho)$-stabilizer of the isotopy class of $c$, and let $\cala=\rho^{-1}(\langle T_c\rangle)$. 

Then $(\calh,\cala)$ satisfies Property~$\qnsep$.
\end{prop}

Proposition~\ref{prop:characterization-vertices} is the combination of our next three lemmas. Lemma~\ref{lemma:property-a} below checks Assertions~(a),(b) and~(c) from Definition~\ref{de:pnsep}. For later convenience, in this lemma, we also allow for separating meridians in the statement. As we are assuming throughout this section that the handlebody $V$ has genus at least $3$, there are three possibilities for a meridian $c$, namely:
\begin{enumerate}
\item $c$ is nonseparating. In this case, we will consider the restriction homomorphism $\Stab_{\Mod^1(V)}(c)\to\Mod(\partial V\setminus c)$, which is well-defined and takes its values in $\Mod^0(\partial V\setminus c)$ by definition of $\Mod^1(V)$ (recall Definition~\ref{de:mod1}). Its kernel is isomorphic to $\mathbb{Z}$: it is equal to the intersection of $\Mod^1(V)$ with the cyclic subgroup $\langle T_c\rangle$ generated by the Dehn twist about $c$.
\item $c$ is separating, and none of the two connected components $\Sigma_1,\Sigma_2$ of $\partial V\setminus c$ is a once-holed torus. In this case, we will consider the restriction homomorphism $\Stab_{\Mod^1(V)}(c)\to\Mod(\Sigma_1)\times\Mod(\Sigma_2)$, which is well-defined and takes its values in $\Mod^0(\Sigma_1\cup\Sigma_2)=\Mod^0(\Sigma_1)\times\Mod^0(\Sigma_2)$. Again, its kernel is isomorphic to $\mathbb{Z}$, in fact equal to the intersection of $\Mod^1(V)$ with the cyclic subgroup $\langle T_c\rangle$.
\item $c$ is separating, and among the two connected components $\Sigma_1,\Sigma_2$ of $\partial V\setminus c$, exactly one, say $\Sigma_2$, is a once-holed torus. In this case, we will consider the restriction homomorphism $\Stab_{\Mod^1(V)}(c)\to\Mod(\Sigma_1)$, which is well-defined and takes its values in $\Mod^0(\Sigma_1)$. Its kernel is isomorphic to $\mathbb{Z}^2$ by Corollary~\ref{cor:once-holed}.
\end{enumerate}
All three possibilities are allowed in the following statement.

\begin{lemma}\label{lemma:property-a}
Let $\calg$ be a measured groupoid, equipped with a strict action-type cocycle $\rho:\calg\to\Mod^1(V)$. Let $c$ be a meridian, and let $\calh$ be the $(\calg,\rho)$-stabilizer of the isotopy class of $c$. Let $\Sigma\subseteq\partial V$ be the union of all components of $\partial V\setminus c$ which are not once-holed tori. Let $A$ be the kernel of the restriction homomorphism $\Stab_{\Mod^1(V)}(c)\to\Mod^0(\Sigma)$, and let $\cala=\rho^{-1}(A)$. 

Then $\calh$ is everywhere nonamenable, $\cala$ is a normal amenable subgroupoid of $\calh$ of infinite type, and every stably normal amenable subgroupoid of $\calh$ is stably contained in $\cala$.
\end{lemma}

\begin{proof}
As follows from the discussion preceding the statement of the lemma, the subsurface $\Sigma$ is nonempty because the genus of $V$ is at least $3$. Lemma~\ref{lemma:pA-in-complement} ensures that $\Stab_{\Mod^1(V)}(c)$ contains a nonabelian free subgroup, so Lemma~\ref{lemma:free} shows that $\calh$ is everywhere nonamenable. 

Normality of $\cala$ in $\calh$ follows from the normality of $A$ in $\Stab_{\Mod^1(V)}(c)$. As mentioned in the discussion preceding the statement, $A$ is amenable: it is either isomorphic to $\mathbb{Z}$ or to $\mathbb{Z}^2$. 
As $\rho$ has trivial kernel, it follows that $\cala$ is amenable (see \cite[Corollary~3.39]{GH}). And $\cala$ is of infinite type because $A$ is infinite and $\rho$ is action-type. 

Let now $\calb\subseteq\calh$ be a stably normal amenable subgroupoid of $\calh$. Let $S\subseteq\Sigma$ be a connected component of $\Sigma$. Let $\rho_S:\calh\to\Mod^0(S)$ be the cocycle obtained by post-composing $\rho$ with the restriction homomorphism. Let also $F\subseteq\Stab_{\Mod^1(V)}(c)$ be a nonabelian free subgroup which embeds into $\Mod^0(S)$ under the restriction homomorphism, and whose image in $\Mod^0(S)$ contains a pseudo-Anosov mapping class (this exists because $S$ is not a once-holed torus, see Lemma~\ref{lemma:pA-in-complement}). Let $\calh'=\rho^{-1}(F)$.

By Lemma~\ref{lemma:crs}, we can find a partition $Y=\dunion_{i\in I}Y_i$ into at most countably many Borel subsets such that for every $i\in I$, the pair $(\calb_{|Y_i},\rho_{S})$ has a canonical reduction set $\calc_i$. As $\calb$ is stably normal in $\calh$, up to refining the above partition, we can assume that for every $i\in I$, the groupoid $\calb_{|Y_i}$ is normal in $\calh_{|Y_i}$. Lemma~\ref{lemma:crs-normal} thus ensures that $\calc_i$ is $(\calh_{|Y_i},\rho_S)$-invariant, so either $\calc_i=\emptyset$ or $\calc_i$ fills $S$. 

Assume towards a contradiction that $\calc_i=\emptyset$ for some $i\in I$ such that $Y_i$ has positive measure. In other words $(\calb_{|Y_i},\rho_{S})$ is irreducible. As $\rho_S$ has trivial kernel in restriction to $\calh'$, and as $\calh'_{|Y_i}$ (which is contained in $\calh_{|Y_i}$) normalizes $\calb_{|Y_i}$, Lemma~\ref{lemma:kida-ia} implies that $\calh'_{|Y_i}$ is amenable. But $F$ is a nonabelian free group and $\rho$ is action-type, so we get a contradiction to Lemma~\ref{lemma:free}.

It follows that for every $i\in I$, there exists a conull Borel subset $Y_i^*\subseteq Y_i$ such that $\rho_{S}(\calb_{|Y_i^*})=\{1\}$. As $S$ was an arbitrary connected component of $\Sigma$, this precisely means that $\calb$ is stably contained in $\cala$. 
\end{proof}

We now check Assertion~(d) from Definition~\ref{de:pnsep}.

\begin{lemma}\label{lemma:property-b}
Let $\calg$ be a measured groupoid over a standard probability space $Y$, equipped with a strict action-type cocycle $\rho:\calg\to\Mod^1(V)$. Let $c$ be a nonseparating meridian, and let $\calh$ be the $(\calg,\rho)$-stabilizer of the isotopy class of $c$. 

If $\calh'$ is a subgroupoid of $\calg$ which is everywhere nonamenable and contains a stably normal amenable subgroupoid of infinite type, and if $\calh$ is stably contained in $\calh'$, then $\calh$ is stably equal to $\calh'$.
\end{lemma}

\begin{proof}
Let $\cala'$ be an amenable subgroupoid of $\calg$ of infinite type which is contained in $\calh'$ and stably normal in $\calh'$. By Lemma~\ref{lemma:crs}, we can find a partition $Y=\sqcup_{i\in I}Y_i$ into at most countably many Borel subsets such that for every $i\in I$, the pair $(\cala'_{|Y_i},\rho)$ has a (possibly empty) canonical reduction set $\calc_i$. For every $i\in I$, we let $X_i$ be the (possibly empty) boundary multicurve of $\calc_i$. As $\cala'$ is stably normal in $\calh'$, up to refining the above partition, we can assume that for every $i\in I$, the set $\calc_i$ is $(\calh'_{|Y_i},\rho)$-invariant (Lemma~\ref{lemma:crs-normal}), and therefore so is the multicurve $X_i$. As $\calh$ is stably contained in $\calh'$, we will also assume up to refining the above partition once more that for every $i\in I$, one has $\calh_{|Y_i}\subseteq\calh'_{|Y_i}$. In particular $X_i$ is $(\calh_{|Y_i},\rho)$-invariant, and since $\rho$ takes its values in $\Mod^1(V)$, any curve component of the multicurve $X_i$ is actually $(\calh_{|Y_i},\rho)$-invariant. This implies that either $X_i=\emptyset$ or $X_i=c$ by Lemma~\ref{lemma:uniqueness}.

Let $i\in I$ be such that $Y_i$ has positive measure. If $X_i=\emptyset$, then as $\cala'$ is of infinite type and $\rho$ has trivial kernel, we deduce that $\calc_i=\emptyset$, i.e.\ $(\cala'_{|Y_i},\rho)$ is irreducible. Lemma~\ref{lemma:kida-ia} then implies that $\calh'_{|Y_i}$ is amenable, a contradiction. Therefore $X_i=c$, so the isotopy class of $c$ is $(\calh'_{|Y_i},\rho)$-invariant. As this is true for every $i\in I$ such that $Y_i$ has positive measure, we deduce that $\calh'$ is stably contained in $\calh$.
\end{proof}

We finally check Assertion~(e) from Definition~\ref{de:pnsep}.

\begin{lemma}
Let $\calg$ be a measured groupoid over a standard probability space $Y$, equipped with a strict action-type cocycle $\rho:\calg\to\Mod^1(V)$. Let $c$ be a nonseparating meridian, and let $\calh$ be the $(\calg,\rho)$-stabilizer of the isotopy class of $c$.

Then for every Borel subset $U\subseteq Y$ of positive measure, the groupoid $\calh_{|U}$ is not contained in any product-like subgroupoid of $\calg_{|U}$.  
\end{lemma}

\begin{proof}
Let $U\subseteq Y$ be a Borel subset of positive measure. Assume towards a contradiction that $\calh_{|U}$ is contained in a product-like subgroupoid $\calp$ of $\calg_{|U}$. Let $\calp_1,\calp_2,\cala_1^1,\cala_1^2,\cala_2^1,\cala_2^2\subseteq\calp$ be as in the definition of a product-like subgroupoid (Definition~\ref{de:product-like}).

Up to restricting to a Borel subset of $U$ of positive measure, we can assume that $(\calp,\rho)$ has a canonical reduction multicurve $X$. As $\calh_{|U}\subseteq\calp$, the isotopy class of $X$ is $(\calh_{|U},\rho)$-invariant, and in fact every component curve of $X$ is $(\calh_{|U},\rho)$-invariant because we are working in $\Mod^1(V)$. So by Lemma~\ref{lemma:uniqueness}, $X=\emptyset$ or $X=c$. The induced cocycle $\rho':\calp\to\Mod^0(\partial V\setminus X)$ is well-defined after possibly restricting to a conull Borel subset of $X$, and its kernel is amenable (it is trivial if $X=\emptyset$, and contained in $\rho^{-1}(\langle T_c\rangle)$ if $X=c$). As $\calp$ is everywhere nonamenable, it follows that $\rho'(\calp_{|U^*})\neq\{1\}$ for every conull Borel subset $U^*\subseteq U$ (as otherwise $\calp_{|U^*}$ would be equal to the kernel and therefore amenable). In particular, the subsurface $\partial V\setminus X$ is active for $(\calp,\rho)$, and therefore $(\calp,\rho')$ is irreducible. As $\calp_2$ is everywhere nonamenable, we also have $\rho'((\calp_2)_{|U'})\neq\{1\}$ for every positive measure Borel subset $U'\subseteq U$. As $\calp_2$ is stably normal in $\calp$, Corollary~\ref{cor:reducibility-normal} therefore ensures that $(\calp_2,\rho')$ is also irreducible. 

By definition of a strongly Schottky pair (Definition~\ref{de:strong-schottky}, applied to $(\cala_1^1,\cala_1^2)$), there exists a Borel subset $U'\subseteq U$ of positive measure such that every normal amenable subgroupoid of $\langle (\cala_1^1)_{|U'},(\cala_1^2)_{|U'}\rangle$ is stably trivial. In particular, the kernel of $\rho'$ restricted to the subgroupoid $\langle (\cala_1^1)_{|U'}, (\cala_1^2)_{|U'}\rangle$ is stably trivial. As $\cala_1^1$ is of infinite type, it follows that for every Borel subset $U''\subseteq U'$ of positive measure, we have $\rho'((\cala_1^1)_{|U''})\neq\{1\}$. As $(\cala_1^1)_{|U'}$ is stably normalized by $(\calp_2)_{|U'}$, Corollary~\ref{cor:reducibility-normal} ensures that $((\cala_1^1)_{|U'},\rho')$ is irreducible. 

By definition of a strongly Schottky pair (applied to $(\cala_2^1,\cala_2^2)$), there exists a Borel subset $W\subseteq U'$ of positive measure such that every normal amenable subgroupoid of $\langle(\cala_2^1)_{|W},(\cala_2^2)_{|W}\rangle$ is stably trivial. In particular, the kernel of $\rho'$ restricted to the subgroupoid $\langle(\cala_2^1)_{|W},(\cala_2^2)_{|W}\rangle$ is stably trivial. This implies that we can find a positive measure Borel subset $W'\subseteq W$ such that $\rho'$ has trivial kernel in restriction to $\langle (\cala_2^1)_{|W'},(\cala_2^2)_{|W'}\rangle$. As $(\cala_1^1)_{|W'}$ is stably normalized by $\langle (\cala_2^1)_{|W'}, (\cala_2^2)_{|W'}\rangle$, it thus follows from Lemma~\ref{lemma:kida-ia} that $\langle (\cala_2^1)_{|W'}, (\cala_2^2)_{|W'}\rangle$ is amenable, which yields the desired contradiction. 
\end{proof}

\subsection{Stabilizers of separating meridians do not satisfy Property~$\pnsep$}

\begin{lemma}\label{lemma:stab-non-unilateral}
Let $\calg$ be a measured groupoid, equipped with a strict action-type cocycle $\rho:\calg\to\Mod^1(V)$, and let $\calh$ be a measured subgroupoid of $\calg$. Let $c$ be a separating meridian, and assume that the isotopy class of $c$ is $(\calh,\rho)$-invariant. 

Then $\calh$ does not satisfy Property~$\pnsep$.
\end{lemma}

\begin{proof}
We first assume that one complementary component $\Sigma$ of $c$ is a once-holed torus. Then $\Sigma$ contains, up to isotopy, a unique nonseparating meridian $d$ (Lemma~\ref{lem:meridian-stab-containment}), so $\calh$ is contained in the $(\calg,\rho)$-stabilizer $\calh'$ of the isotopy class of $d$. In addition $\calh'$ is everywhere nonamenable and contains $\rho^{-1}(\langle T_d\rangle)$ as a normal amenable subgroupoid of infinite type. Finally $\calh'$ is not stably contained in $\calh$ because $\partial V\setminus d$ supports a pseudo-Anosov handlebody group element $g$ (Lemma~\ref{lemma:pA-in-complement}), and no nontrivial power of $g$ preserves the isotopy class of $c$. So Assumption~(d) from Definition~\ref{de:pnsep} fails. 

We now assume that both complementary components $\Sigma_1,\Sigma_2$ of $c$ have genus at least $2$. Let $\calp$ be the $(\calg,\rho)$-stabilizer of $c$. Then $\calh$ is contained in $\calp$ (up to restricting to a conull Borel subset of the base space $Y$), and we will prove that $\calp$ is product-like (which will imply that Assumption~(e) from Definition~\ref{de:pnsep} fails). For every $i\in\{1,2\}$, let $P_i$ be the subgroup of $\Mod^1(V)$ made of elements that have a representative supported in $\Sigma_i$, and let $\calp_i=\rho^{-1}(P_i)$. Then $\calp_i$ is normal in $\calp$. For every $i\in\{1,2\}$, let $f_i^1$ and $f_i^2$ be two elements of $P_i$ that generate a nonabelian free subgroup of $\Mod^1(V)$, see e.g.\ Lemma~\ref{lemma:pA-in-complement} for their existence. For every $i\in\{1,2\}$ and every $j\in\{1,2\}$, let $\cala_i^j=\rho^{-1}(\langle f_i^j\rangle)$. Then $\cala_i^j$ is normalized by $\calp_{3-i}$, and Lemma~\ref{lemma:free} ensures that $(\cala_i^1,\cala_i^2)$ is a strongly Schottky pair of subgroupoids of $\calg$. This completes our proof. 
\end{proof}

\subsection{Admissible decorated multicurves and their active subgroups}\label{sec:hermitages}

A \emph{decorated multicurve} is a pair $(X,\mathfrak{A})$, where $X$ is a multicurve on $\partial V$, and $\mathfrak{A}$ is a subset of the set of complementary components of $X$ in $\partial V$. We make the following definition, which uses the restriction homomorphisms for rotationless mapping classes; compare Section~\ref{sec:bg-stab}.
\begin{de}\label{de:hermitage}
  Let $(X,\mathfrak{A})$ be a decorated multicurve. The \emph{active
    subgroup $A$ of $(X,\mathfrak{A})$} is the maximal subgroup of
  $\Mod^1(V)$ satisfying
  \begin{enumerate}[a)]
  \item each $a \in A$ preserves $X$.
  \item for each complementary component $S$ of $A$ which is not contained in $\mathfrak{A}$, the image
    of the restriction homomorphism
    \[ A \to \Mod(S) \]
    is trivial. 
  \end{enumerate}
  The decorated multicurve $(X,\mathfrak{A})$ is \emph{admissible} if
  its active subgroup $A$ is amenable, $X$ is the canonical reduction multicurve
  of $A$, and $\mathfrak{A}$ is its set of active complementary
  components.
\end{de}

We give the simplest example of this definition, which suffices for
our purposes.  Suppose that $X = \{\delta\}$ is a single meridian
$\delta$ on $\partial V$, and put $\mathfrak{A}=\emptyset$. Then the
active subgroup of $(X,\mathfrak{A})$ is generated by a power of the
twist $T_\delta$. In particular $(X,\mathfrak{A})$ is admissible.

However, if the genus of $V$ is at least $3$, then with no choice of
$\mathfrak{A} \neq \emptyset$ do we obtain an admissible decorated
multicurve, since in that case the active subgroup will always contain
a nonabelian free group.

Similarly, if $X = \{\alpha_1, \alpha_2\}$ is an annulus pair, then
$(X, \emptyset)$ is admissible, for the same reason as in the meridian
case.

\bigskip Let $\calg$ be a measured groupoid over a standard probability space $Y$, equipped with an action-type cocycle $\rho:\calg\to\Mod^1(V)$. We say that a pair $(\calh,\cala)$ of subgroupoids of $\calg$ is \emph{admissible} with respect to $\rho$ if there exist a conull Borel subset $Y^*\subseteq Y$ and a partition $Y^*=\dunion_{i\in I}Y_i$ into at most countably many Borel subsets, such that for every $i\in I$, there exist a multicurve $X_i$ on $\partial V$, and a subset $\mathfrak{A}_i$ of the set of all complementary components of $X_i$ such that $(X_i,\mathfrak{A}_i)$ is admissible, $\calh_{|Y_i}$ is equal to the $(\calg_{|Y_i},\rho)$-stabilizer of the isotopy class of $X_i$, and denoting by $A_i\subseteq\Mod^1(V)$ the active subgroup of $(X_i,\mathfrak{A}_i)$, one has $\cala_{|Y_i}=\rho^{-1}(A_i)_{|Y_i}$. Notice that, although the above partition is not unique (one can always pass to a further partition), the map sending $y\in Y_i$ to the isotopy class of $(X_i,\mathfrak{A}_i)$ is uniquely determined by $(\calh,\cala)$, up to changing its value on a conull Borel subset (indeed $X_i$ is recovered as the canonical reduction multicurve of $(\cala_{|Y_i},\rho)$, and $\mathfrak{A}_i$ as its active subsurface). We call it the \emph{decomposition map} of $(\calh,\cala)$.

\begin{lemma}\label{lemma:product-like}
Let $\calg$ be a measured groupoid over a standard probability space $Y$, equipped with a strict action-type cocycle $\rho:\calg\to\Mod^1(V)$. Let $\cala,\calh$ be measured subgroupoids of $\calg$, with $\cala\subseteq\calh$. 

If $(\calh,\cala)$ satisfies Property~$\qnsep$, then $(\calh,\cala)$ is an admissible pair.
\end{lemma}

\begin{proof}
By Assumption~(b) from Definition~\ref{de:pnsep}, the groupoid $\cala$ is amenable, of infinite type, and stably normal in $\calh$. Up to a countable partition of the base space $Y$, we will assume that $\cala$ is normal in $\calh$. Up to a further partition, we can also assume that $(\cala,\rho)$ has a canonical reduction set $\calc$ (Lemma~\ref{lemma:crs}). Let $X$ be the boundary multicurve of $\calc$, let $\mathfrak{A}$ be the set of all active complementary components for $(\cala,\rho)$, and let $\overline{\mathfrak{A}}$ be the set of all complementary components of $X$ not in $\mathfrak{A}$. Up to replacing $Y$ by a conull Borel subset, we will assume using Lemma~\ref{lemma:crs-normal} that $\rho(\calh)\subseteq\Stab_{\Mod^1(V)}(X)$.

We will first prove that $(X,\mathfrak{A})$ is admissible, so let us assume towards a contradiction that it is not. Let $A\subseteq\Mod^1(V)$ be the active subgroup of $(X,\mathfrak{A})$. Then there exists a conull Borel subset $Y^*\subseteq Y$ such that $\rho(\cala_{|Y^*})\subseteq A$. Therefore $\calc$ is exactly the set of all curves whose isotopy class is $A$-invariant, so $X$ is the canonical reduction multicurve of $A$ and $\mathfrak{A}$ is its set of active complementary components. Therefore, our assumption that $(X,\mathfrak{A})$ is not admissible implies that $A$ is not amenable, so it contains a nonabelian free subgroup $F$ (by the Tits alternative for mapping class groups \cite{McCa,Iva-book}).
 
Let $\Sigma_1$ be the union of all subsurfaces in $\mathfrak{A}$, viewed as a (possibly disconnected) surface of finite type. Let $\rho_1:\calh\to\Mod^0(\Sigma_1)$ be the cocycle obtained by composing $\rho$ with the restriction to $\Sigma_1$. We now observe that for every $U\subseteq Y$ of positive measure, the restriction to $U$ of the kernel of $\rho_1$  is nontrivial: otherwise, as $(\cala_{|U},\rho_1)$ is irreducible and $\calh_{|U}$ normalizes $\cala_{|U}$, Lemma~\ref{lemma:kida-ia} ensures that $\calh_{|U}$ is amenable, a contradiction to Assumption~(a) from Definition~\ref{de:pnsep}. 

Let $\calb$ be the kernel of $\rho_1$. The groupoid $\calb$ is normal in $\calh$. We first assume that $\calb$ is amenable, and reach a contradiction in this case. Assumption~(c) from Definition~\ref{de:pnsep} ensures that there exists a Borel subset $U\subseteq Y$ of positive measure such that $\calb_{|U}\subseteq\cala_{|U}$. But the $\rho$-image of every element of $\calb_{|U}$ acts trivially on all components in $\mathfrak{A}$, while the $\rho$-image of every element of $\cala_{|U}$ acts trivially on all components in $\overline{\mathfrak{A}}$. It follows that for every $g\in\calb_{|U}$, the element $\rho(g)$ is a multitwist around curves in $X$. As $\rho$ has trivial kernel and $\calb_{|U}$ is nontrivial, it follows that the subgroup $\mathrm{Tw}$ of $\Mod^1(V)$ consisting of all multitwists about the curves in $X$ is infinite. Let $\calh'=\rho^{-1}(\Stab_{\Mod^1(V)}(X))$. Then $\calh_{|U}\subseteq\calh'_{|U}$, and $\calh'_{|U}$ is everywhere nonamenable (it contains $\calh_{|U}$) and contains $\rho^{-1}(\mathrm{Tw})_{|U}$ as a normal amenable subgroupoid of infinite type. So Assumption~(d) from Definition~\ref{de:pnsep} ensures that there exists a Borel subset $U'\subseteq U$ of positive measure such that $\calh'_{|U'}=\calh_{|U'}$. Now, the groupoid $\rho^{-1}(F)_{|U'}$ is contained in $\calh_{|U'}$, so it normalizes $\cala_{|U'}$, and up to changing $U'$ to a positive measure subset, $\rho_1$ has trivial kernel in restriction to $\rho^{-1}(F)_{|U'}$ (this uses Lemma~\ref{lemma:free} and the fact that the kernel of $\rho_1$ is a normal amenable subgroupoid). As $(\cala_{|U'},\rho_1)$ is irreducible, Lemma~\ref{lemma:kida-ia} implies that $\rho^{-1}(F)_{|U'}$ is amenable, a contradiction to Lemma~\ref{lemma:free}.  
   
We now assume that $\calb$ is nonamenable, and also reach a contradiction in this case. As $\rho$ has trivial kernel, the subgroup $P_2$ of $\Mod^1(V)$ made of all elements that fix the isotopy class of $X$ and act trivially on all connected components in $\mathfrak{A}$ is nonamenable, and therefore contains a nonabelian free subgroup. Let $P=\Stab_{\Mod^1(V)}(X)$, and let $\calp=\rho^{-1}(P)$ (i.e.\ $\calp=\calh'$ with the notation from above). We will now reach a contradiction to Assumption~(e) from Definition~\ref{de:pnsep} by proving that $\calp$ is a product-like subgroupoid of $\calg$ (in which $\calh$ is contained). 

Let $P_1\unlhd P$ be the normal subgroup made of all elements of $P$ that act trivially on all components in $\overline{\mathfrak{A}}$ (i.e.\ $P_1=A$), and recall that $P_2\unlhd P$ is the normal subgroup made of all elements of $P$ acting trivially on all components in $\mathfrak{A}$. Then $\calp_i=\rho^{-1}(P_i)$ is normal in $\calp=\rho^{-1}(P)$ for every $i\in\{1,2\}$. Notice that $P_1$ contains the nonabelian free subgroup $F$, and we saw in the previous paragraph that $P_2$ also contains a nonabelian free subgroup. For every $i\in\{1,2\}$, let $A_i^1,A_i^2$ be two cyclic subgroups of $P_i$ that generate a nonabelian free subgroup, and for $j\in\{1,2\}$, let $\cala_i^j=\rho^{-1}(A_i^j)$. As $P_1$ and $P_2$ centralize each other, it follows that each $\cala_i^j$ is normalized by $\calp_{3-i}$. In addition, Lemma~\ref{lemma:free} ensures that $(\cala_i^1,\cala_i^2)$ is a strongly Schottky pair of subgroupoids of $\calg$. So $\calp$ is a product-like subgroupoid of $\calg$, which is the desired contradiction. 

This contradiction shows that $(X,\mathfrak{A})$ is admissible. Now, let $\cala'=\rho^{-1}(A)$, and let $\calh'$ be the $(\calg,\rho)$-stabilizer of the isotopy class of $X$. Then $\calh$ is contained in $\calh'$, and $\calh'$ contains $\cala'$ as a normal amenable subgroupoid of infinite type. So Assertion~(d) from Definition~\ref{de:pnsep} ensures that $\calh$ is stably equal to $\calh'$. And Assertion~(c) then implies that $\cala$ is stably equal to $\cala'$. This proves that $(\calh,\cala)$ is an admissible pair.
\end{proof}

\subsection{Compatibility}

Two decorated multicurves $(X,\mathfrak{A})$ and $(X',\mathfrak{A}')$ are \emph{compatible} if $X$ and $X'$ are disjoint up to isotopy, and given any two components $S\in\mathfrak{A}$ and $S'\in\mathfrak{A}'$, either $S$ and $S'$ are isotopic, or they are disjoint up to isotopy. We start with the following observation. 

\begin{lemma}\label{lemma:hermitage-disjoint}
Let $(X,\mathfrak{A})$ and $(X',\mathfrak{A}')$ be two admissible decorated multicurves, with respective active subgroups $A,A'$.

If $(X,\mathfrak{A})$ and $(X',\mathfrak{A}')$ are compatible, then $\langle A,A'\rangle$ is amenable.
\end{lemma}

\begin{proof}
  
  Let $\Sigma$ be the union of all subsurfaces in $\mathfrak{A}$ and
  all annuli around curves in $X$. Let $\Sigma'$ be the union of all
  subsurfaces in $\mathfrak{A}'$ and all annuli around curves in
  $X'$. Let $S$ be the union of all subsurfaces in
  $\mathfrak{A}\cap\mathfrak{A}'$.

  We therefore have containments
  $\langle A,A'\rangle_{|\Sigma\setminus S}\subseteq
  A_{|\Sigma\setminus S}$ and
  $\langle A,A'\rangle_{|\Sigma'\setminus S}\subseteq
  A'_{\Sigma'\setminus S}$. Since $A, A'$ are amenable by
  admissability, all these restrictions are amenable. By
  Lemma~\ref{lem:ivanov}, they are in fact abelian.

  Let $B=\langle A,A'\rangle_{|S}$, a subgroup of $\Mod(S)$. We
  claim that $B$ is amenable. Indeed, every element in $B$ can be
  written as the restriction to $S$ of an element of $\Mod(V)$ of the
  form $\varphi\varphi'\psi$, where $\varphi\in\Mod(\partial V)$ is
  supported on $\Sigma\setminus S$, where
  $\varphi'\in\Mod(\partial V)$ is supported on $\Sigma'\setminus S$,
  and $\psi\in\Mod(\partial V)$ is supported on $S$. Since
  $\Sigma\setminus S$ and $\Sigma'\setminus S$ can be realized
  disjointly, the commutator of any two elements of this
  form is an element of $\Mod(V)$ that acts trivially on
  $\Sigma\setminus S$ and on $\Sigma'\setminus S$. In other words, we
  have proved that every element in $[B,B]$ is the restriction of a
  handlebody element supported on $S$, and therefore contained in
  $A\cap A'$. By admissibility, $A\cap A'$ is amenable, so $[B,B]$ is
  amenable, and therefore $B$ is amenable.

  Now the map
  $\varphi\mapsto (\varphi_{|S},\varphi_{|\Sigma\setminus
    S},\varphi_{|\Sigma'\setminus S})$ determines a homomorphism from
  $\langle A,A'\rangle$ to
  $B\times A_{|\Sigma\setminus S}\times A'_{|\Sigma'\setminus S}$,
  with abelian kernel (consisting of multitwists), so
  $\langle A,A'\rangle$ is amenable.
%
\end{proof}

The converse is also true.
\begin{lemma}\label{lemma:hermitage-nondisjoint}
  Let $(X,\mathfrak{A})$ and $(X',\mathfrak{A}')$ be two admissible decorated
  multicurves, with respective active subgroups $A,A'$. 

  If $(X,\mathfrak{A})$ and $(X',\mathfrak{A}')$ are not compatible,
  then $\langle A,A'\rangle$ contains a nonabelian free group, and is
  in particular non-amenable.
\end{lemma}

\begin{proof}  
  Since $(X,\mathfrak{A})$ and $(X',\mathfrak{A}')$ are not
  compatible, either $X, X'$ are not disjoint, or there are components
  $S \in \mathfrak{A}, S' \in \mathfrak{A}'$ which are neither equal
  nor disjoint (up to isotopy).

  First suppose that $X, X'$ are not disjoint. Since $X$ (respectively
  $X'$) is the canonical reduction system for $A$ (respectively $A'$),
  this gives infinite order elements $a \in A, a' \in A'$ no powers of
  which commute (since their canonical reduction systems intersect,
  see Lemma~\ref{lem:commute-disjoint-crs}).

  Similarly, if $X, X'$ are disjoint, but there are components
  $S \in \mathfrak{A}, S' \in \mathfrak{A}'$ which are neither equal
  or disjoint, we can find such elements. Indeed, these can be chosen
  to restrict to pseudo-Anosov homeomorphisms in $S, S'$: such
  elements exist because $S$ and $S'$ are active, and by admissibility
  $X,X'$ are the canonical reduction multicurves of $A,A'$.

  But any two non-commuting, infinite order elements $a,a'$ of the
  mapping class group have powers which generate a free group. This is
  a well-known folklore result, but can e.g.\ be found in the
  literature by using \cite[Theorem~1.8]{Kob} to show that suitably
  high powers $a^n, (a')^n$ embed into a nonabelian right-angled Artin
  group, and noncommuting elements there have powers that generate a
  free group (e.g.\ by \cite[Lemma~3.1]{Kob}, as explained in the paragraph below \cite[Theorem~1.8]{Kob}.
\end{proof}

Let $\calg$ be a measured groupoid over a standard probability space $Y$ which admits an action-type cocycle $\rho:\calg\to\Mod^1(V)$. Two admissible pairs $(\calh,\cala)$ and $(\calh',\cala')$ (with respect to $\rho$) are \emph{compatible with respect to $\rho$} if, denoting by $(X,\mathfrak{A})$ and $(X',\mathfrak{A}')$ their respective decomposition maps, for a.e.\ $y\in Y$, the pairs $(X(y),\mathfrak{A}(y))$ and $(X'(y),\mathfrak{A}'(y))$ are compatible. The following proposition gives a purely groupoid-theoretic characterization of compatibility (i.e.\ with no reference to the cocycle $\rho$).  

\begin{prop}\label{prop:characterization-edges}
Let $\calg$ be a measured groupoid over a standard probability space $Y$, equipped with a strict action-type cocycle $\rho:\calg\to\Mod^1(V)$. Let $(\calh,\cala)$ and $(\calh',\cala')$ be two admissible pairs with respect to $\rho$. Then the following are equivalent.
\begin{enumerate}[(1)]
\item $(\calh,\cala)$ and $(\calh',\cala')$ are compatible with respect to $\rho$;
\item for every Borel subset $U\subseteq Y$ of positive measure, there exists a Borel subset $V\subseteq U$ of positive measure such that $\langle\cala_{|V},\cala'_{|V}\rangle$ is amenable.  
\end{enumerate}
\end{prop}

\begin{proof}[{Proof of Proposition~\ref{prop:characterization-edges}}]
Let $Y^*=\dunion_{i\in I}Y_i$ be a countable Borel partition of a conull Borel subset $Y^*\subseteq Y$ such that for every $i\in I$, there exist admissible pairs $(X_i,\mathfrak{A}_i)$ and $(X'_i,\mathfrak{A}'_i)$ such that $\calh_{|Y_i}=\rho^{-1}(\Stab_{\Mod^1(V)}(X_i))$ and $\calh'_{|Y_i}=\rho^{-1}(\Stab_{\Mod^1(V)}(X'_i))$, and letting $A_i,A'_i\subseteq\Mod^1(V)$ be the active subgroups of $(X_i,\mathfrak{A}_i)$ and $(X'_i,\mathfrak{A}'_i)$ respectively, we have $\cala_{|Y_i}=\rho^{-1}(A_i)$ and $\cala'_{|Y_i}=\rho^{-1}(A'_i)$.

We first prove that $\neg (1) \Rightarrow \neg (2)$. If $(1)$ fails, then there exists $i_0\in I$ such that $Y_{i_0}$ has positive measure and $(X_{i_0},\mathfrak{A}_{i_0})$ and $(X'_{i_0},\mathfrak{A}'_{i_0})$ are not compatible. Then there exist $g_{i_0}\in A_{i_0}$ and $g'_{i_0}\in A'_{i_0}$ that generate a nonabelian free subgroup of $\Mod^1(V)$ by Lemma~\ref{lemma:hermitage-nondisjoint}. Lemma~\ref{lemma:free} ensures that for every Borel subset $V\subseteq Y_{i_0}$ of positive measure, the groupoid $\langle\rho^{-1}(\langle g_{i_0}\rangle)_{|V},\rho^{-1}(\langle g'_{i_0}\rangle)_{|V}\rangle$ is nonamenable. Therefore $\langle \cala_{|V},\cala'_{|V}\rangle$ is also nonamenable for every Borel subset $V\subseteq Y_{i_0}$ of positive measure, so $(2)$ fails.

We now prove that $(1)\Rightarrow (2)$. If $(1)$ holds, then for every $i\in I$ such that $Y_i$ has positive measure, the pairs $(X_i,\mathfrak{A}_i)$ and $(X'_i,\mathfrak{A}'_i)$ are compatible, so $\langle A_i,A'_i\rangle$ is amenable (Lemma~\ref{lemma:hermitage-disjoint}). Let now $U\subseteq Y$ be a Borel subset of positive measure, and let $V\subseteq U$ of positive measure be contained in $Y_{i_0}$ for some $i_0\in I$. Then  $\langle\cala_{|V},\cala'_{|V}\rangle$ is contained in $\rho^{-1}(\langle A_{i_0},A'_{i_0}\rangle)_{|V}$, which is amenable because $\langle A_{i_0},A'_{i_0}\rangle$ is and $\rho$ has trivial kernel (see \cite[Corollary~3.39]{GH})
\end{proof}

Let $\calh,\calh'$ be two measured subgroupoids of $\calg$ of meridian type with respect to an action-type cocycle $\rho:\calg\to\Mod^1(V)$. We say that $\calh$ and $\calh'$ are \emph{compatible with respect to $\rho$} if, denoting by $\varphi,\varphi'$ their respective meridian maps with respect to $\rho$, for a.e.\ $y\in Y$, the meridians $\varphi(y)$ and $\varphi'(y)$ are disjoint up to isotopy.

\begin{cor}\label{cor:adjacency-preserved}
Let $\calg$ be a measured groupoid over a standard probability space $Y$, equipped with two strict action-type cocycles $\rho_1,\rho_2:\calg\to\Mod^1(V)$, and let $\calh,\calh'$ be two measured subgroupoids of $\calg$ of meridian type with respect to both $\rho_1$ and $\rho_2$.

Then $\calh$ and $\calh'$ are compatible with respect to $\rho_1$ if and only if they are compatible with respect to $\rho_2$.
\end{cor}

\begin{proof}
Let $Y^*=\dunion_{j\in J}Y_j$ be a partition of a conull Borel subset $Y^*\subseteq Y$ into at most countably many Borel subsets such that for every $i\in\{1,2\}$ and every $j\in J$, there exist meridians $c_{i,j},c'_{i,j}$ such that $\calh_{|Y_j},\calh'_{|Y_j}$ are equal to the $(\calg_{|Y_j},\rho_i)$-stabilizers of the isotopy classes of $c_{i,j},c'_{i,j}$, respectively. 

For every $i\in\{1,2\}$ and every $j\in J$, let $A_{i,j}$ (resp.\ $A'_{i,j}$) be the subgroup of $\Mod^1(V)$ made of all elements that act trivially in restriction to every connected component of $\partial V\setminus c_{i,j}$ (resp. $\partial V\setminus c'_{i,j}$) which is not a once-holed torus. Notice that $A_{i,j},A'_{i,j}$ are the active subgroups of some admissible decorated multicurves $(X_{i,j},\mathfrak{A}_{i,j}), (X'_{i,j},\mathfrak{A}'_{i,j})$, by letting $X_{i,j}$ and $X'_{i,j}$ be obtained from $c_{i,j}$ and $c'_{i,j}$ by adding the unique nonseparating meridian in every complementary component which is a once-holed torus, and letting $\mathfrak{A}_{i,j}=\mathfrak{A}'_{i,j}=\emptyset$. See the examples right after Definition~\ref{de:hermitage}. Notice that $c_{i,j}$ and $c'_{i,j}$ are disjoint up to isotopy if and only if $(X_{i,j},\emptyset)$ and $(X'_{i,j},\emptyset)$ are compatible.

For every $i\in\{1,2\}$, let $\cala_i\subseteq\calh$ be a subgroupoid such that $(\cala_i)_{|Y_j}=\rho_i^{-1}(A_{i,j})_{|Y_j}$ for every $j\in J$, and let $\cala'_i\subseteq\calh'$ be defined in the same way, using $A'_{i,j}$ in place of $A_{i,j}$. Then $(\calh,\cala_i)$ and $(\calh',\cala'_i)$ are admissible pairs with respect to $\rho_i$. Lemma~\ref{lemma:property-a} thus ensures that $\cala_1$ and $\cala_2$ are stably equal (as they are both stably maximal for the property of being a stably normal amenable subgroupoid of $\calh$), and likewise $\cala'_1$ and $\cala'_2$ are stably equal. The conclusion therefore follows from Proposition~\ref{prop:characterization-edges}.
\end{proof}

\subsection{Characterizing subgroupoids of nonseparating-meridian type}\label{sec:characterization-nonseparating}

The goal of this section is to prove the following proposition.

\begin{prop}\label{prop:characterize-nonseparating}
Let $\calg$ be a measured groupoid over a standard probability space $Y$, equipped with two strict action-type cocycles $\rho_1,\rho_2:\calg\to\Mod^1(V)$, and let $\calh\subseteq\calg$ be a measured subgroupoid.

Then $\calh$ is of nonseparating-meridian type with respect to $\rho_1$ if and only if it is of nonseparating-meridian type with respect to $\rho_2$. 
\end{prop}

A decorated multicurve $(X,\mathfrak{A})$ is \emph{clean} if it is not of the form $(c,\emptyset)$ for some separating meridian $c$. The \emph{graph of clean admissible decorated multicurves} $\mathbb{M}$ is the graph whose vertices correspond to isotopy classes of clean admissible decorated multicurves, where two distinct vertices are joined by an edge if the corresponding decorated multicurves are compatible. The \emph{graph of nonseparating meridians} $\mathbb{D}^{\nsep}$ is the graph whose vertices correspond to isotopy classes of nonseparating meridians, where two distinct vertices are joined by an edge if the corresponding meridians are disjoint up to isotopy. Notice that $\mathbb{D}^{\nsep}$ is naturally a subgraph of $\mathbb{M}$, by sending a nonseparating meridian $c$ to the pair $(c,\emptyset)$.

\begin{lem}\label{lemma:recognize-annulus}
Every injective graph map\footnote{i.e.\ preserving adjacency and non-adjacency} from $\mathbb{D}^{\nsep}$ to $\mathbb{M}$ takes its values in $\mathbb{D}^{\nsep}$ (viewed as a subgraph of $\mathbb{M}$ via the natural inclusion).  
\end{lem}

\begin{proof}
Given a subset $F\subseteq V\mathbb{D}^{\nsep}$, we denote by $\mathrm{Lk}_{\mathbb{D}^{\nsep}}(F)$ the link of $F$ in $\mathbb{D}^{\nsep}$, i.e.\ the set of all vertices of $\mathbb{D}^{\nsep}$ which are at distance $1$ from every vertex of $F$. 

Let $v\in V(\mathbb{D}^{\nsep})$ be a vertex. By completing $v$ to a pair of pants decomposition made of nonseparating meridians, we can find $3g-3$ pairwise distinct, pairwise adjacent vertices $v=v_1,\dots,v_{3g-3}$ (corresponding to the pants decomposition) such that for every $i\in\{1,\dots,3g-3\}$, one has 
\[ \mathrm{Lk}_{\mathbb{D}^{\nsep}}(\{v_1,\dots,v_{3g-3}\}) \subsetneq \mathrm{Lk}_{\mathbb{D}^{\nsep}}(\{v_1,\dots,v_{i-1},v_{i+1},\dots,v_{3g-3}\})\setminus\{v_i\}. \]
This holds because the leftmost term is empty, while the rightmost term is infinite, and consists in all nonseparating meridians contained in the $4$-holed sphere created by removing $v_i$ from the collection.  
	
So the same strict inclusion of links should hold for their images in $\mathbb{M}$, which correspond to pairwise compatible decorated multicurves $(X_1,\mathfrak{A}_1),\dots,(X_{3g-3},\mathfrak{A}_{3g-3})$. For every $i\in\{1,\dots,3g-3\}$, we let $\Sigma_i$ be the subsurface of $\partial V$ equal to the union of all subsurfaces in $\mathfrak{A}_i$, together with all annuli around curves in $X_i$ that are not boundary curves of any subsurface in $\mathfrak{A}_i$. Notice that the set $\{\Sigma_1,\dots,\Sigma_{3g-3}\}$ cannot contain both a subsurface $S$ and the collar neighborhood $A$ of one of its boundary components, as otherwise removing $A$ from the collection does not change the link.
	 More generally, for every $i\in\{1,\dots,3g-3\}$, one of the connected components $\Sigma'_i\subseteq\Sigma_i$ is not a connected component of some $\Sigma_j$ with $j\neq i$, and is also not the collar neighborhood of a boundary curve of $\Sigma_j$ -- otherwise removing $(X_i,\mathfrak{A}_i)$ from the collection does not change its link. So the subsurfaces $\Sigma'_1,\dots,\Sigma'_{3g-3}$ are pairwise nonisotopic and pairwise disjoint, and $\{\Sigma'_1,\dots,\Sigma'_{3g-3}\}$ does not contain a subsurface together with the collar neighborhood of one of its boundary components. For every $i\in\{1,\dots,3g-3\}$, let $\{b_{i,1},\dots,b_{i,k_i}\}$ be the set of all boundary curves of $\Sigma'_i$, and let $\{d_{i,1},\dots,d_{i,\ell_i}\}$ be a set of isotopy classes of essential simple closed curves on $\Sigma'_i$ that form a pair of pants decomposition of $\Sigma'_i$ (with the convention that in the case of an annulus, the former set contains two isotopic curves, and the latter set is empty). The tuple consisting of all $b_{i,j}$ and $d_{i,j}$ contains at least $6g-6$ curves which are pairwise disjoint, each being repeated at most twice up to isotopy (and the $d_{i,j}$ are not isotopic to any other curve in the collection). So every subsurface $\Sigma'_i$ contributes exactly two curves that are both of the form $c_{i,j}$, and is therefore an annular subsurface. Furthermore, since there are $3g-3$ such, and no $\Sigma'_i$ appears as a subsurface of $\Sigma_j, j\neq i$, we actually have $\Sigma'_i = \Sigma_i$ for all $i$. Therefore $(X_i,\mathfrak{A}_i)=(c_i,\emptyset)$, where $c_i$ is the core curve of the annulus $\Sigma_i$. As $(X_i,\mathfrak{A}_i)$ is admissible, some power of the twist around $c_i$ must belong to the handlebody group, so $c_i$ is a meridian by \cite[Theorem~1.11]{Oer} or \cite[Theorem~1]{McC}. As $(X_i,\mathfrak{A}_i)$ is clean, the meridian $c_i$ is nonseparating, and the conclusion follows.
\end{proof}

\begin{proof}[Proof of Proposition~\ref{prop:characterize-nonseparating}]
Let $v\in V(\mathbb{D}^{\nsep})$, in other words $v$ is the isotopy class of a nonseparating meridian. Let $\calh_v$ be the $(\calg,\rho_1)$-stabilizer of $v$, and let $\cala_v=\rho_1^{-1}(\langle T_v\rangle)$. Then $(\calh_v,\cala_v)$ satisfies Property~$\qnsep$ (by Proposition~\ref{prop:characterization-vertices}, applied to the cocycle $\rho_1$). Lemma~\ref{lemma:product-like}, applied to the cocycle $\rho_2$, implies that $(\calh_v,\cala_v)$ is an admissible pair with respect to $\rho_2$. So there exist a conull Borel subset $Y^*\subseteq Y$ and a partition $Y^*=\dunion_{i\in I} Y_{v,i}$ into at most countably many Borel subsets of positive measure such that for every $i\in I$, there exists a (unique) admissible pair $(X_{v,i},\mathfrak{A}_{v,i})$ such that $(\calh_v)_{|Y_{v,i}}$ is the $(\calg_{|Y_{v,i}},\rho_2)$-stabilizer of $X_{v,i}$ and, denoting by $A_{v,i}$ the active subgroup of $(X_{v,i},\mathfrak{A}_{v,i})$, one has $(\cala_v)_{|Y_{v,i}}=\rho^{-1}(A_{v,i})_{|Y_{v,i}}$. In addition, Lemma~\ref{lemma:stab-non-unilateral} ensures that $(X_{v,i},\mathfrak{A}_{v,i})$ is clean. 
For every $y\in Y$ and every $v\in V(\mathbb{D}^{\nsep})$, we then let $\theta(y,v)=(X_{v,i},\mathfrak{A}_{v,i})$ whenever $y\in Y_{v,i}$. This defines a Borel map $\theta:Y\times V(\mathbb{D}^{\nsep})\to V(\mathbb{M})$. 

We claim that for almost every $y\in Y$, the map $\theta(y,\cdot)$ determines an injective graph map $\mathbb{D}^{\nsep}\to\mathbb{M}$. Let us first explain how to complete the proof of the proposition from this claim. By Lemma~\ref{lemma:recognize-annulus}, every injective graph map $\mathbb{D}^{\nsep}\to\mathbb{M}$ sends nonseparating meridians to nonseparating meridians. Therefore, if $v$ is a nonseparating meridian, then $X_{v,i}$ is a nonseparating meridian (and $\mathfrak{A}_{v,i}=\emptyset$) whenever $Y_{v,i}$ has positive measure, and the proposition follows.

We are now left with proving the above claim. First, for almost every $y\in Y$, the map $\theta(y,\cdot)$ is injective. Indeed otherwise, as $V(\mathbb{D}^{\nsep})$ is countable, there exist a Borel subset $U\subseteq Y$ of positive measure and two non-isotopic nonseparating meridians $c,c'$ such that for every $y\in U$, one has $\theta(y,c)=\theta(y,c')$ (we denote by $(X,\mathfrak{A})$ the common image). In particular, the $(\calg_{|U},\rho_1)$-stabilizer of $c$ is stably equal to the $(\calg_{|U},\rho_1)$-stabilizer of $c'$, since they are both stably equal to the $(\calg_{|U},\rho_2)$-stabilizer of $X$. This contradicts Lemma~\ref{lemma:uniqueness}.

Second, Proposition~\ref{prop:characterization-edges} ensures that for almost every $y\in Y$, the map $\theta(y,\cdot)$ is a graph map, i.e.\ it preserves both adjacency and non-adjacency. 
\end{proof}

\subsection{Characterizing subgroupoids of separating-meridian type}\label{sec:characterization-separating}

In this section, we establish a purely groupoid-theoretic characterization of subgroupoids of separating-meridian type with respect to a strict action-type cocycle towards $\Mod^1(V)$, and derive that being of separating-meridian type is a notion that does not depend on the choice of such a cocycle.

\subsubsection{Property~$(\mathrm{P}_{\mathrm{sep}})$}

We proved in Proposition~\ref{prop:characterize-nonseparating} that for a subgroupoid $\calh\subseteq\calg$, being of nonseparating meridian type does not depend of the choice of an action-type cocycle $\calg\to\Mod^1(V)$. Also, it follows from Corollary~\ref{cor:adjacency-preserved} that compatibility of two subgroupoids of nonseparating meridian type is also independent of such a choice. Thus, the following notion is a purely groupoid-theoretic property.

\begin{de}[Property~$(\mathrm{P}_{\mathrm{sep}})$]
Let $\calg$ be a measured groupoid over a standard probability space $Y$ which admits a strict action-type cocycle towards $\Mod^1(V)$. A measured subgroupoid $\calh\subseteq\calg$ satisfies \emph{Property~$(\mathrm{P}_{\mathrm{sep}})$} if
\begin{enumerate}
\item $\calh$ contains a strongly Schottky pair of subgroupoids;
\item there exists a stably normal amenable subgroupoid $\calb\subseteq\calh$ of infinite type, such that for every measured subgroupoid $\calh'\subseteq\calg$ of nonseparating-meridian type, and every stably normal amenable subgroupoid $\cala\subseteq\calh'$, the intersection $\cala\cap\calb$ is stably trivial;
\item given any two subgroupoids $\calh_1,\calh_2\subseteq\calg$ of nonseparating-meridian type, and any Borel subset $U\subseteq Y$ of positive measure, assuming that $\calh_{|U}\subseteq (\calh_1\cap\calh_2)_{|U}$, then $(\calh_1)_{|U}$ and $(\calh_2)_{|U}$ are stably equal;
\item there exist $3g-4$ measured subgroupoids $\calh_1,\dots,\calh_{3g-4}$ of $\calg$ of nonseparating-meridian type, which are pairwise compatible, such that
\begin{enumerate}
\item for every Borel subset $U\subseteq Y$ of positive measure, and any two distinct $i,j\in\{1,\dots,3g-4\}$, one has $(\calh_i)_{|U}\neq (\calh_j)_{|U}$, and 
\item for every $j\in\{1,\dots,3g-4\}$, every stably normal amenable subgroupoid $\cala_j$ of $\calh_j$ is stably contained in $\calh$.
\end{enumerate}
\end{enumerate}
\end{de}

Notice that this notion is stable under restriction to a positive measure subset. Also, if $\calh$ is a subgroupoid of $\calg$, and if there exists a partition $Y=\dunion_{i\in I}Y_i$ into at most countably many Borel subsets such that for every $i\in I$, the subgroupoid $\calh_{|Y_i}$ of $\calg_{|Y_i}$ satisfies Property~$(\mathrm{P}_{\mathrm{sep}})$, then $\calh$ (as a subgroupoid of $\calg$) satisfies Property~$(\mathrm{P}_{\mathrm{sep}})$. 

To motivate the definition, we begin by proving that subgroupoids of separating meridian type have this property.

\begin{prop}\label{prop:characterise-separating-part1}
  Let $\calg$ be a measured groupoid over a standard probability space
  $Y$, equipped with a strict action-type cocycle
  $\rho:\calg\to\Mod^1(V)$. Let $c$ be a separating meridian, and let $\calh$ be the $(\calg,\rho)$-stabilizer of the isotopy class of $c$. 
  
  Then $\calh$ satisfies Property~$(\mathrm{P}_{\mathrm{sep}})$.
\end{prop}

\begin{proof}
  For Property~$\psep(1)$, Lemma~\ref{lemma:pA-in-complement} ensures that the stabilizer in $\Mod^1(V)$
  of any separating meridian contains a nonabelian free subgroup
  (recall that we are assuming that the genus of $V$ is at least
  $3$). Lemma~\ref{lemma:free} thus implies that $\calh$ contains a
  strongly Schottky pair of subgroupoids.

  For Property~$\psep(2)$, notice that the (intersection of $\Mod^1(V)$ with the) cyclic subgroup
  $\langle T_{c}\rangle$ generated by the Dehn twist about $c$ is
  normal in the stabilizer of $c$. Therefore
  $\calb=\rho^{-1}(\langle T_c\rangle)$ is a normal subgroupoid
  of $\calh$, which is amenable as $\rho$ has trivial kernel,
  and of infinite type because $\rho$ is action-type. 
  Let now $\calh'$ be a measured
  subgroupoid of $\calg$ of nonseparating meridian type, and let
  $\cala\subseteq\calh'$ be a stably normal amenable subgroupoid of
  infinite type. By Lemma~\ref{lemma:property-a}, we can find a partition $Y^*=\dunion_{i\in I}Y_i$ of a conull Borel subset $Y^*\subseteq Y$ into at most countably many Borel subsets such that for every $i\in I$, there exists a
  nonseparating meridian $d_i$ such that
  $\cala_{|Y_i}\subseteq\rho^{-1}(\langle T_{d_i}\rangle)_{|Y_i}$. It
  follows that $(\cala\cap\calb)_{|Y_i}$ is trivial, so
  $\cala\cap\calb$ is stably trivial.

  Property~$\psep(3)$ follows from the fact that for every separating
  meridian $c$, there is at most one nonseparating meridian $d$ fixed
  by every element of $\Stab_{\Mod^1(V)}(c)$ (in fact, the existence
  of such a $d$ occurs precisely when one of the two connected
  components of $\partial V\setminus c$ is a once-holed torus, in
  which case it contains a unique nonseparating meridian up to
  isotopy, and we take $d$ as such -- notice that we are using the
  fact that the genus of $V$ is at least $3$ here; see
  Lemma~\ref{lem:meridian-stab-containment}).

  We now prove that $\calh$ satisfies Property~$\psep(4)$. Let $\{c_1,\dots,c_{3g-4}\}$ be a set of $3g-4$
  pairwise non-isotopic nonseparating meridians which together with
  $c$ form a pair of pants decomposition of $\partial V$. For every
  $j\in\{1,\dots,3g-4\}$, let $\calh_j$ be the $(\calg,\rho)$-stabilizer
  of the isotopy class of $c_{j}$. Then $\calh_1,\dots,\calh_{3g-4}$
  are of nonseparating meridian type. Lemma~\ref{lemma:uniqueness}
  ensures that they satisfy Assertion~$(4.a)$. Finally,
  Lemma~\ref{lemma:property-a} ensures that every stably normal
  amenable subgroupoid $\cala_j$ of $\calh_j$ is stably contained in
  $\rho^{-1}(\langle T_{c_{j}}\rangle)$. In particular each
  $\cala_j$ is stably contained in $\calh$.
\end{proof}

\subsubsection{Characterization}

Our goal is now to characterize subgroups of separating meridian-type by proving the following proposition.

\begin{prop}\label{prop:characterise-separating}
Let $\calg$ be a measured groupoid over a standard probability space $Y$, equipped with a strict action-type cocycle $\rho:\calg\to\Mod^1(V)$. Let $\calh$ be a measured subgroupoid of $\calg$. The following assertions are equivalent.
\begin{enumerate}
\item The subgroupoid $\calh$ is of separating-meridian type with respect to $\rho$.
\item The subgroupoid $\calh$ satisfies Property~$(\mathrm{P}_{\mathrm{sep}})$, and is stably maximal among all measured subgroupoids of $\calg$ with respect to this property, i.e.\ if $\calh'$ is another subgroupoid satisfying Property~$\psep$, and if $\calh$ is stably contained in $\calh'$, then $\calh$ is stably equal to $\calh'$.
\end{enumerate}
\end{prop}
 
Before turning to the proof of Proposition~\ref{prop:characterise-separating}, we record the following consequence.

\begin{cor}\label{cor:separating-type}
Let $\calg$ be a measured groupoid over a base space $Y$, equipped with two strict action-type cocycles $\rho_1,\rho_2:\calg\to\Mod^1(V)$, and let $\calh\subseteq\calg$ be a measured subgroupoid. 

Then $\calh$ is of separating-meridian type with respect to $\rho_1$ if and only if it is of separating-meridian type with respect to $\rho_2$. \qed
\end{cor}

Our goal is now to prove Proposition~\ref{prop:characterise-separating}. Our first lemma exploits the first two assumptions of Property~$(\mathrm{P}_{\mathrm{sep}})$ in order to derive information about the possible canonical reduction multicurves of $\calb$.

\begin{lemma}\label{lemma:crs-for-special-meridian}
Let $\calg$ be a measured groupoid over a standard probability space $Y$, equipped with a strict action-type cocycle $\rho:\calg\to\Mod^1(V)$. Let $\calb,\calh$ be measured subgroupoids of $\calg$, with $\calb\subseteq\calh$. Assume that 
\begin{enumerate}
\item $\calh$ contains a strongly Schottky pair of subgroupoids;
\item $\calb$ is amenable and of infinite type, and stably normal in $\calh$;
\item for every measured subgroupoid $\calh'\subseteq\calg$ of nonseparating-meridian type, and every stably normal amenable subgroupoid $\cala\subseteq\calh'$, the intersection $\cala\cap\calb$ is stably trivial.
\end{enumerate}
Then for every Borel subset $U\subseteq Y$ of positive measure, the pair $(\calb_{|U},\rho)$ cannot have a canonical reduction multicurve consisting of a single nonseparating meridian.
\end{lemma}

\begin{proof}
Assume towards a contradiction that $(\calb_{|U},\rho)$ has a canonical reduction multicurve which is reduced to a single nonseparating meridian $c$. As $\calb$ is stably normal in $\calh$, up to restricting to a positive measure Borel subset of $U$, we can assume that $c$ is $(\calh_{|U},\rho)$-invariant (Lemma~\ref{lemma:crs-normal}). In particular, letting $\Sigma=\partial V\setminus c$, we have a natural cocycle $\rho':\calh_{|U}\to\Mod^0(\Sigma)$. In view of the description of curve stabilizers recalled in Section~\ref{sec:bg}, the kernel of $\rho'$ is contained in $\rho^{-1}(\langle T_c\rangle)_{|U}$. As $\rho$ has trivial kernel, it follows that the kernel of $\rho'$ is amenable. In particular, letting $(\cala^1,\cala^2)$ be a strongly Schottky pair of subgroupoids of $\calh$, there exists a positive measure Borel subset $V\subseteq U$ such that $\rho'$ has trivial kernel when restricted to $\langle \cala^1_{|V},\cala^2_{|V}\rangle$.

Our third assumption, applied by taking for $\calh'$ the $(\calg,\rho)$-stabilizer of $c$, and with $\cala=\rho^{-1}(\langle T_c\rangle)$, ensures that $\calb\cap\rho^{-1}(\langle T_c\rangle)$ is stably trivial. Let $W\subseteq V$ be a Borel subset of positive measure such that $(\calb\cap \rho^{-1}(\langle T_c\rangle))_{|W}$ is trivial. Then $\rho'$ also has trivial kernel when restricted to $\calb_{|W}$. In particular $(\calb_{|W},\rho')$ is irreducible, and Lemma~\ref{lemma:kida-ia} implies that $\langle\cala^1_{|W},\cala^2_{|W}\rangle$ is amenable, a contradiction.   
\end{proof}

\begin{lemma}\label{lemma:separating}
Let $\calg$ be a measured groupoid over a standard probability space $Y$, equipped with a strict action-type cocycle $\rho:\calg\to\Mod^1(V)$. Let $\calh$ be a measured subgroupoid of $\calg$ which satisfies Property~$(\mathrm{P}_{\mathrm{sep}})$. 

Then there exists a partition $Y=\dunion_{i\in I}Y_i$ into at most countably many Borel subsets such that for every $i\in I$, there exists an $(\calh_{|Y_i},\rho)$-invariant isotopy class of separating meridian. 
\end{lemma}

\begin{proof}
Let $\calh_1,\dots,\calh_{3g-4}$ be subgroupoids of $\calg$ provided by Property~$(\mathrm{P}_{\mathrm{sep}})(4)$. Up to partitioning $Y$ into at most countably many Borel subsets, we can assume that for every $j\in\{1,\dots,3g-4\}$, the groupoid $\calh_j$ is equal to the $(\calg,\rho)$-stabilizer of the isotopy class of a nonseparating meridian $d_j$. 

Let $\calb\subseteq\calh$ be as in Property~$(\mathrm{P}_{\mathrm{sep}})(2)$. By Lemma~\ref{lemma:crs}, we can find a partition $Y=\dunion_{i\in I}Y_i$ into at most countably many Borel subsets of positive measure such that for every $i\in I$, the pair $(\calb_{|Y_i},\rho)$ has a canonical reduction set $\calc_i$, with boundary multicurve $X_i$. As $\calb$ is stably normal in $\calh$, up to refining this partition, we can assume that for every $i\in I$, the isotopy class of the multicurve $X_i$ is $(\calh_{|Y_i},\rho)$-invariant. 

We first observe that for every $i\in I$, one has $\calc_i\neq\emptyset$. Indeed, otherwise $(\calb_{|Y_i},\rho)$ is irreducible. As $\calb$ is amenable and stably normal in $\calh$, and $\rho$ has trivial kernel, Lemma~\ref{lemma:kida-ia} implies that $\calh_{|Y_i}$ is amenable, contradicting Property~$(\mathrm{P}_{\mathrm{sep}})(1)$.

For every $j\in\{1,\dots,3g-4\}$, let $T_j$ be the Dehn twist about the meridian $d_j$. Then $\cala_j=\rho^{-1}(\langle T_j\rangle)$ is a normal amenable subgroupoid of $\calh_j$. Property~$(\mathrm{P}_{\mathrm{sep}})(4.b)$ thus ensures that $\cala_j$ is stably contained in $\calh$. Therefore, for every curve $c$ in $X_i$, there exists a positive integer $k$ such that the isotopy class of $c$ is fixed by $T_j^k$. This implies that $X_i$ is disjoint (up to isotopy) from all meridians $d_j$. 

We now claim that for every $i\in I$, the multicurve $X_i$ contains at most one of the curves $d_j$. Indeed, assume by contradiction that it contains two curves $d_{j_1}$ and $d_{j_2}$. Then $\calh_{|Y_i}$ is contained in $(\calh_{j_1}\cap \calh_{j_2})_{|Y_i}$. As $\calh_{j_1}$ and $\calh_{j_2}$ are of nonseparating meridian type with respect to $\rho$, Property~$(\mathrm{P}_{\mathrm{sep}})(3)$ implies that there exists a positive measure Borel subset $U\subseteq Y_i$ such that $(\calh_{j_1})_{|U}=(\calh_{j_2})_{|U}$, contradicting Property~$(\mathrm{P}_{\mathrm{sep}})(4.a)$. 

As $\{d_1,\dots,d_{3g-4}\}$ is a set of $3g-4$ pairwise disjoint and pairwise non-isotopic nonseparating simple closed curves on $\partial V$, one of the complementary components of the union of all curves $d_j$ is a $4$-holed sphere $S$. Notice that every essential simple closed curve contained in $S$ is a meridian, and $X_i$ may contain such a curve. 

At this point we know that $X_i$ contains at most one of the curves $d_1,\dots,d_{3g-4}$ and is disjoint from them up to isotopy, so any other curve in $X_i$ must be contained in the complementary $4$-holed sphere. This leaves three possibilities for the canonical reduction multicurve of $(\calb_{|Y_i},\rho)$, namely:
\begin{enumerate}
\item a single nonseparating meridian (either one of the meridians $d_j$, or else a nonseparating meridian contained in $S$);
\item the union of a nonseparating meridian $d_j$ and a nonseparating essential simple closed curve (in fact a meridian) contained in $S$;
\item a separating (on $\partial V$) essential simple closed curve (in fact a meridian) contained in $S$, possibly together with a meridian $d_j$.
\end{enumerate} 
The first case is excluded by Lemma~\ref{lemma:crs-for-special-meridian}, the second case is excluded using Property~$(\mathrm{P}_{\mathrm{sep}})(3)$ and Lemma~\ref{lemma:uniqueness}, and the last case leads to the desired conclusion of our lemma.
\end{proof}

\begin{proof}[Proof of Proposition~\ref{prop:characterise-separating}]
We first prove that $(1)\Rightarrow (2)$. Let $\calh$ be a measured subgroupoid of $\calg$ of separating meridian type with respect to $\rho$, and let $Y^*=\dunion_{i\in I} Y_i$ be a partition of a conull Borel subset $Y^*\subseteq Y$ into at most countably many Borel subsets, such that for every $i\in I$, the groupoid $\calh_{|Y_i}$ is equal to the $(\calg_{|Y_i},\rho)$-stabilizer of the isotopy class of a separating meridian $c_i$.

Proposition~\ref{prop:characterise-separating-part1} implies that $\calh$ satisfies Property~$\psep$. We need to check that $\calh$ is stably maximal among all measured subgroupoids of $\calg$ that satisfy Property~$\psep$. So let $\calh'$ be a measured subgroupoid of $\calg$ which satisfies Property~$\psep$, and such that $\calh$ is stably contained in $\calh'$. By Lemma~\ref{lemma:separating}, up to refining the above partition of $Y$, we can assume that for every $i\in I$, there exists a separating meridian $c'_i$ whose isotopy class is $(\calh'_{|Y_i},\rho)$-invariant. Lemma~\ref{lemma:uniqueness-disk} implies that $c_i=c'_i$ for every $i\in I$. It follows that $\calh'$ is stably contained in $\calh$, so they are stably equal. This completes our proof of the implication $(1)\Rightarrow (2)$.

We now prove that $(2)\Rightarrow (1)$, so let $\calh$ be a measured groupoid of $\calg$ that satisfies Assertion~$(2)$. By Lemma~\ref{lemma:separating}, we can find a partition $Y=\dunion_{i\in I}Y_i$ into at most countably many Borel subsets such that for every $i\in I$, there exists a separating meridian $c_i$ whose isotopy class is $(\calh_{|Y_i},\rho)$-invariant. Let $\calh'$ be a measured subgroupoid of $\calg$ such that for every $i\in I$, the groupoid $\calh'_{|Y_i}$ is equal to the $(\calg_{|Y_i},\rho)$-stabilizer of the isotopy class of $c_i$. Then $\calh$ is stably contained in $\calh'$. In addition, $\calh'$ is of separating meridian type, so Proposition~\ref{prop:characterise-separating-part1} shows that $\calh'$ satisfies Property~$(\mathrm{P}_{\mathrm{sep}})$. The maximality assumption on $\calh$ therefore implies that $\calh$ is stably equal to $\calh'$. Hence $\calh$ itself is of separating meridian type, which concludes our proof.  
\end{proof}

\subsection{Conclusion}

Before concluding the proof of our main theorem, we first record the following easy consequence of Propositions~\ref{prop:characterize-nonseparating} and~\ref{prop:characterise-separating}.

\begin{prop}\label{prop:meridian-type}
Let $\calg$ be a measured groupoid, equipped with two strict action-type cocycles $\rho_1,\rho_2:\calg\to\Mod^1(V)$, and let $\calh\subseteq\calg$ be a measured subgroupoid. 

Then $\calh$ is of meridian type with respect to $\rho_1$ if and only if it is of meridian type with respect to $\rho_2$.
\qed
\end{prop}

We will now simply say that $\calh$ is \emph{of meridian type} to mean that it is of meridian type with respect to any action-type cocycle $\calg\to\Mod^1(V)$. We are now in position to complete the proof of Theorem~\ref{theo:main-2}, which as we already explained at the beginning of this section yields the measure equivalence superrigidity of handlebody groups in genus at least $3$.

\begin{proof}[Proof of Theorem~\ref{theo:main-2}]
Let $\calg$ be a measured groupoid over a standard probability space $Y$, and let $\rho_1,\rho_2:\calg\to\Mod^1(V)$ be two strict action-type cocycles. Let $\mathbb{D}$ be the disk graph of $V$: we recall that its vertices are the isotopy classes of meridians in $\partial V$, and two such isotopy classes are joined by an edge if they have disjoint representatives.
  
Proposition~\ref{prop:meridian-type} ensures that for every vertex $v\in V(\mathbb{D})$, there exists a Borel map $\phi_{v}:Y\to V(\mathbb{D})$ such that for every $w\in V(\mathbb{D})$, letting $Y_{v,w}=\phi_v^{-1}(w)$, the $(\calg_{|Y_{v,w}},\rho_1)$-stabilizer of $v$ is stably equal to the $(\calg_{|Y_{v,w}},\rho_2)$-stabilizer of $w$. Lemmas~\ref{lemma:uniqueness} and~\ref{lemma:uniqueness-disk} ensure that the map $\phi_v$ is essentially unique. 

For every $y\in Y$ and every $v\in V(\mathbb{D})$, we then let $\psi(y,v)=\phi_{v}(y)$. This defines a Borel map $\psi:Y\times V(\mathbb{D})\to V(\mathbb{D})$. 

We claim that for a.e.\ $y\in Y$, the map $\psi(y,\cdot)$ is a graph automorphism of $\mathbb{D}$. Indeed, injectivity follows from the same argument as in the proof of Proposition~\ref{prop:characterize-nonseparating}, and the fact that $\psi(y,\cdot)$ is almost everywhere a graph map follows from Corollary~\ref{cor:adjacency-preserved}. We now show that for almost every $y\in Y$, the map $\psi(y,\cdot)$ is surjective. So let $c$ be a meridian. By Proposition~\ref{prop:meridian-type}, there exists a Borel partition of a conull Borel subset $Y^*\subseteq Y$ into at most countably many Borel subsets $Y_i$ such that for every $i$, the $(\calg_{|Y_i},\rho_2)$-stabilizer of $c$ coincides with the $(\calg_{|Y_i},\rho_1)$-stabilizer of some $c_i\in V(\mathbb{D})$. It follows that $\psi(y,c_i)=c$ for almost every $y\in Y_i$. Surjectivity follows.

By the main theorem of \cite{KS}, the natural map $\mathrm{Mod}^\pm(V)\to\mathrm{Aut}(\mathbb{D})$ is
an isomorphism (noting again that the genus of $V$ is at least $3$). We can thus find a
Borel map $\theta:Y \to \mathrm{Mod}^\pm(V)$ so that for a.e.\ $y\in Y$ we have that 
$\psi(y,\delta) = \theta(y)(\delta)$ for all meridians.

We are left with showing that $\theta$ satisfies the equivariance condition required in Theorem~\ref{theo:main-2}. This amounts to proving that there exists a conull Borel subset $Y^*\subseteq Y$ such that for every $g\in\calg_{|Y^*}$ and every vertex $v\in V(\mathbb{D})$, one has $\psi(r(g),\rho_1(g)v)=\rho_2(g)\psi(s(g),v).$ As $\calg$ is a countable union of bisections, it is enough to prove it for almost every $g$ in a bisection $B$ (inducing a Borel isomorphism between $U=s(B)$ and $V=r(B)$). Up to further partitioning $B$, we can assume that $(\rho_1)_{|B}$ and $(\rho_2)_{|B}$ are constant, with values $\gamma_1,\gamma_2$, and that $\psi(\cdot,v)_{|U}$ is constant, with value $w$. We now aim to show that for almost every $y\in V$, one has $\psi(y,\gamma_1 v)=\gamma_2w$. By definition of $\psi$, the $(\calg_{|U},\rho_1)$-stabilizer of $v$ is stably equal to the $(\calg_{|U},\rho_2)$-stabilizer of $w$. Conjugating by the bisection, it follows that the $(\calg_{|V},\rho_1)$-stabilizer of $\gamma_1 v$ is stably equal to the $(\calg_{|V},\rho_2)$-stabilizer of $\gamma_2 w$, which is exactly what we wanted to show.  
\end{proof}

\section{Applications}

\subsection{Lattice embeddings and automorphisms of the Cayley graph}

A first consequence of our work is that handlebody groups do not admit any interesting lattice embeddings in locally compact second countable groups.

\begin{theo}\label{theo:lattice}
Let $V$ be a handlebody of genus at least $3$. Let $G$ be a locally compact second countable group, equipped with its (left or right) Haar measure. Let $\Gamma$ be a finite index subgroup of $\Mod^\pm(V)$, and let $\sigma:\Gamma\to G$ be an injective homomorphism whose image is a lattice. 

Then there exists a homomorphism $\theta:G\to\Mod^{\pm}(V)$ with compact kernel such that for every $f\in\Gamma$, one has $\theta\circ\sigma(f)=f$.
\end{theo}

\begin{proof}
Theorem~\ref{theo:main-2} precisely says that $\Mod^{\pm}(V)$ is rigid with respect to action-type cocycles in the sense of \cite[Definition~4.1]{GH}. As $\Mod^{\pm}(V)$ is ICC (Lemma~\ref{lemma:icc}), the theorem thus follows from \cite[Theorem~4.7]{GH}.\footnote{Theorem~4.7 from \cite{GH} records works of Furman \cite{Fur} and Kida \cite[Theorem~8.1]{Kid}. The idea behind its proof is that the lattice embedding $\sigma$ determines a self measure equivalence coupling of $\Gamma$ (acting on $G$ equipped with its Haar measure), and the rigidity statement provided by Theorem~\ref{theo:main-2} from the present paper ensures that the self coupling of $\Gamma$ on $G$ factors through the obvious coupling on $\Mod^{\pm}(V)$ where $\Gamma$ acts by left/right multiplication. This yields a Borel map $G\to\Mod^{\pm}(V)$, and some extra work is needed to upgrade it to a continuous homomorphism.}
\end{proof}

A theorem of Suzuki ensures that $\Mod^{\pm}(V)$ is finitely generated \cite{Suz} (it is in fact finitely presented by work of Wajnryb \cite{Waj}). Given a finitely generated group $G$ and a finite generating set $S$ of $G$, the \emph{Cayley graph} $\Cay(G,S)$ is defined as the simple graph whose vertices are the elements of $G$, with an edge between distinct elements $g,h$ if $g^{-1}h\in S\cup S^{-1}$.

\begin{theo}
Let $V$ be a handlebody of genus at least $3$. 
\begin{enumerate}
\item For every finite generating set $S$ of $\Mod^{\pm}(V)$, every automorphism of $\Cay(\Mod^{\pm}(V),S)$ is at bounded distance from the left multiplication by an element of $\Mod^{\pm}(V)$.
\item For every torsion-free finite-index subgroup $\Gamma\subseteq\Mod^{\pm}(V)$ and every finite generating set $S'$ of $\Gamma$, the automorphism group of $\Cay(\Gamma,S')$ is countable (in fact it embeds as a subgroup of $\Mod^{\pm}(V)$ containing $\Gamma$). 
\end{enumerate}
\end{theo}

\begin{proof}
Using the fact that $\Mod^\pm(V)$ is ICC, this follows from Theorem~\ref{theo:main-2} and \cite[Corollary~4.8]{GH} (the idea behind the proof is to view $\Mod^\pm(V)$ as a cocompact lattice in the automorphism group of its Cayley graph and apply the previous theorem).
\end{proof}

As mentioned in the introduction, torsion-freeness of $\Gamma$ is crucial in the second conclusion in view of \cite[Lemma~6.1]{dlST}.

\subsection{Orbit equivalence rigidity and von Neumann algebras}\label{sec:oe}

Seminal work of Furman \cite{Fur3} has shown that measure equivalence rigidity is intimately related to orbit equivalence rigidity of ergodic group actions. In fact two countable groups are measure equivalent if and only if they admit stably orbit equivalent free measure-preserving ergodic actions by Borel automorphisms on standard probability spaces, see \cite[Proposition~6.2]{Gab-l2}. 

\paragraph*{Orbit equivalence rigidity.}

Let $\Gamma_1$ and $\Gamma_2$ be two countable groups, and for every $i\in\{1,2\}$, let $(X_i,\mu_i)$ be a standard probability space equipped with a free ergodic measure-preserving action of $\Gamma_i$. 

The actions $\Gamma_1\actson X_1$ and $\Gamma_2\actson X_2$ are \emph{virtually conjugate} (as in \cite[Definition~1.3]{Kid-oe}) if there exist finite normal subgroups $F_i\unlhd \Gamma_i$, finite-index subgroups $Q_i\subseteq \Gamma_i/F_i$, and free ergodic measure-preserving actions $Q_i\actson Y_i$ on standard probability spaces, so that $Q_1\actson Y_1$ and $Q_2\actson Y_2$ are conjugate, and for every $i\in\{1,2\}$, the action of $\Gamma_i/F_i$ on $X_i/F_i$ is induced from the $Q_i$-action on $Y_i$. This implies in particular that the groups $\Gamma_1$ and $\Gamma_2$ are virtually isomorphic (i.e.\ commensurable up to finite kernels). 

The following is a weaker notion. The actions $\Gamma_1\actson X_1$ and $\Gamma_2\actson X_2$ are \emph{stably orbit equivalent} if there exist positive measure Borel subsets $A_1\subseteq X_1$ and $A_2\subseteq X_2$ and a measure-scaling isomorphism $\theta:A_1\to A_2$\footnote{in other words $\theta$ induces a measure space isomorphism between the probability spaces $\frac{1}{\mu_1(A_1)}A_1$ and $\frac{1}{\mu_2(A_2)}A_2$} such that for almost every $x\in A_1$, one has $$\theta((\Gamma_1\cdot x)\cap A_1)=(\Gamma_2\cdot \theta(x))\cap A_2.$$ 

A free ergodic measure-preserving action of $\Gamma$ on a standard probability space $X$ is \emph{OE-superrigid} if for every countable group $\Gamma'$, and every free ergodic measure-preserving action of $\Gamma'$ on a standard probability space $X'$, if the $\Gamma$-action on $X$ is stably orbit equivalent to the $\Gamma'$-action on $X'$, then the two actions are virtually conjugate (in particular $\Gamma$ and $\Gamma'$ are virtually isomorphic).

The following theorem follows from our work in the exact same way as for mapping class groups of surfaces \cite{Kid-oe} (see also \cite[Lemma~4.18]{Fur-survey}).

\begin{theo}\label{theo:oe}
Let $V$ be a handlebody of genus at least $3$. Then every free ergodic measure-preserving action of $\Mod^\pm(V)$ on a standard probability space is OE-superrigid.
\end{theo}

\paragraph*{Rigidity of von Neumann algebras.}

Let $\Gamma$ be a countable group, and let $X$ be a standard probability space equipped with a standard ergodic action of $\Gamma$. Associated to the $\Gamma$-action on $X$ is a von Neumann algebra $L^\infty(X)\rtimes\Gamma$, obtained from the Murray--von Neumann construction \cite{MvN}.

We refer the reader to the work of Ozawa and Popa  \cite[Definition~3.1]{OP} for the notion of a \emph{weakly compact} group action. Let us only mention here that these include \emph{profinite} actions, i.e.\ those obtained as inverse limits of actions on finite probability spaces (see \cite[Proposition~3.2]{OP}). For example, this applies to the action of any residually finite countable group on its profinite completion, equipped with the Haar measure. As a subgroup of $\Mod(\partial V)$, the handlebody group $\Mod(V)$ is residually finite by a theorem of Grossman \cite{Gro}. 

A free ergodic measure-preserving action of a countable group $\Gamma$ on a standard probability space $X$ is \emph{$W^\ast_{wc}$-superrigid} if for every countable group $\Gamma'$, and every weakly compact free ergodic measure-preserving action of $\Gamma'$ on a standard probability space $X'$, if the von Neumann algebras $L^\infty(X)\rtimes \Gamma$ and $L^\infty(Y)\rtimes \Gamma'$ are isomorphic, then the $\Gamma$-action on $X$ is virtually conjugate to the $\Gamma'$-action on $X'$. 

\begin{theo}
Let $V$ be a handlebody of genus at least $3$. Then every free ergodic measure-preserving action of $\Mod^\pm(V)$ on a standard probability space is $W^*_{wc}$-superrigid. 
\end{theo}

\begin{proof}
Let $X$ be a standard probability space equipped with a free ergodic measure-preserving action of $\Mod^\pm(V)$, and let $X'$ be a standard probability space equipped with a weakly compact free ergodic measure-preserving action of a countable group $\Gamma'$. Assume that there exists an isomorphism $\theta:L^\infty(X)\rtimes \Mod^\pm(V)\to L^\infty(X')\rtimes\Gamma'$. By \cite[Theorem~7]{HHL}, the group $\Mod^\pm(V)$ is properly proximal in the sense of Boutonnet, Ioana and Peterson \cite{BIP}. It thus follows from \cite[Theorem~1.4]{BIP} that up to unitary conjugacy, the isomorphism $\theta$ sends $L^\infty(X)$ to $L^\infty(X')$. This implies that the actions $\Gamma\actson X$ and $\Gamma'\actson X'$ are orbit equivalent (see \cite{Sin}), so the conclusion follows from the orbit equivalence rigidity statement provided by Theorem~\ref{theo:oe}. 
\end{proof}

\begin{rk}
Beyond the weakly compact case, the only kwown $W^*$-superrigidity result for handlebody groups concerns their Bernoulli actions, that is, actions of the form $\Mod^\pm(V)\actson X_0^{\Mod^\pm(V)}$, where $X_0$ is a standard probability space not reduced to a point, and the action is by shift. More precisely, when $V$ has genus at least 3, if a Bernoulli action $\Mod^\pm(V)\actson X$ and a free, ergodic, probability measure-preserving action of a countable group have isomorphic von Neumann algebras, then the actions are conjugate. This follows from \cite[Theorem~A.2]{HH3}, based on work of Ioana, Popa and Vaes \cite[Theorem~10.1]{IPV}, applied by letting $\Gamma_0$ be the cyclic subgroup generated by a Dehn twist about a nonseparating meridian $\alpha$, letting $\Gamma_1$ be the stabilizer of the isotopy class of $\alpha$, and $\Gamma=\Mod^\pm(V)$. Indeed, to check that \cite[Theorem~A.2]{HH3} applies, we only need to find an element $g\in\Mod^\pm(V)$ such that $g\Gamma_1g^{-1}\cap\Gamma_1$ is infinite, and $\langle\Gamma_1,g\rangle$ generates $\Mod^\pm(V)$. For this, let $\beta,\gamma$ be nonseparating meridians such that $\alpha,\beta,\gamma$ are pairwise disjoint, pairwise non-isotopic, and have connected complement. Let $g\in\Mod^\pm(V)$ be an element sending $\alpha$ to $\beta$ and commuting with the twist $T_\gamma$. Then $g\Gamma_1 g^{-1}\cap\Gamma_1$ is infinite because it contains $T_\gamma$. And $\Mod^\pm(V)$ is generated by $\Gamma_1$ and $g$ because the simplicial graph with vertices the isotopy classes of nonseparating meridians, and edges the nonseparating pairs, is connected (as easily follows from the connectivity of the disk graph) with quotient a single edge.
\end{rk}

\footnotesize

\bibliographystyle{alpha}
\bibliography{ME-Handlebody-bib}

\begin{flushleft}
Sebastian Hensel\\
Mathematisches Institut der Universität München\\
D-80333 München\\
\emph{e-mail: }\texttt{hensel@math.lmu.de}\\
~
\end{flushleft}

\begin{flushleft}
Camille Horbez\\ 
Universit\'e Paris-Saclay, CNRS,  Laboratoire de math\'ematiques d'Orsay, 91405, Orsay, France \\
\emph{e-mail:}\texttt{camille.horbez@universite-paris-saclay.fr}\\[8mm]
\end{flushleft}

\end{document}